\documentclass{scrartcl}

\usepackage[utf8]{inputenc}
\usepackage[T1]{fontenc}

\usepackage{graphicx}
\usepackage{amsmath,amssymb,amsthm}
\usepackage{esint} 
\usepackage{enumitem,color}
\usepackage{comment}
\usepackage{hyperref}
\usepackage{diagbox}
\usepackage{tikz}

\usepackage[style=alphabetic,maxalphanames=4,minalphanames=3,maxbibnames=99]{biblatex}
\AtEveryBibitem{\clearfield{issn}}
\newtheorem{thm}{Theorem}[section]
\newtheorem{prop}[thm]{Proposition}
\newtheorem{lem}[thm]{Lemma}
\newtheorem{cor}[thm]{Corollary}
\theoremstyle{definition}

\theoremstyle{remark}
\newtheorem{rem}[thm]{Remark}

\newcommand{\Z}{\mathbb{Z}}
\newcommand{\R}{\mathbb{R}}
\newcommand{\C}{\mathbb{C}}

\DeclareMathOperator{\Prob}{Prob}
\DeclareMathOperator{\Li}{Li}

\newcommand{\cQ}{\mathcal Q}
\newcommand{\Qs}{Q^{\mathrm{sim}}}
\newcommand{\Ks}{K^{\mathrm{sim}}}
\newcommand{\xs}{x_{\mathrm{s}}}
\newcommand{\ys}{y_{\mathrm{s}}}
\newcommand{\zs}{z_{\mathrm{s}}}
\newcommand{\hatKs}{\hat K^{\mathrm{sim}}}

\newcommand{\ve}{\varepsilon}
\newcommand{\vr}{\varrho}

\DeclareMathOperator{\Airy}{Ai}

\title{%
  Distance statistics of block-weighted planar quadrangulations}
\author{J\'er\'emie Bouttier\thanks{Sorbonne Université and Université Paris Cité, CNRS, IMJ-PRG, F-75005 Paris, France} \and
  Emmanuel Guitter\thanks{Université Paris-Saclay, CNRS, CEA, Institut de physique
    théorique, 91191, Gif-sur-Yvette, France} \and 
  Hugo Manet\thanks{Université Paris Cité, CNRS, IRIF, F-75013, Paris, France}}
\date{\today}

\begin{document}
\maketitle

\begin{abstract}
  We study random planar quadrangulations
  in which each block -- i.e., each simple (without multiple edges) component --
  is assigned a weight $u$.
  We derive an explicit expression for 
  the distance-dependent two-point function,
  defined as the generating function of such block-weighted maps with two marked edges at a fixed graph distance.
  Using contour integral representations in two variables combined with a delicate saddle-point analysis,
  we compute the associated distance profile
  in the scaling limit of large quadrangulations.
  We recover the known phase transition at $u = 9/5$,
  characterized by distinct scaling exponents and different scaling functions,
  below, at, and above criticality.
  We also discuss the block distance profile,
  where the two marked edges are conditioned to lie within the same block.
\end{abstract}

\tableofcontents

\section{Introduction}
\label{sec:intro}

\subsection{Context and motivation}

The study of random planar maps and their large size limits
has been a constant subject of interest over the last decades,
with a myriad of combinatorial as well as probabilistic results discovered over the years; see
e.g.\ \cite{LeGall2014, Schaeffer15, Curien2023, Budd2023} and references therein.
In recent works \cite{Fleurat2024,ZSPhD},
some attention was paid to the statistical properties of so-called
\emph{block-weighted} planar maps which we may describe as follows.

It has long been recognized that many families of planar maps have
natural decompositions into blocks arranged into a tree-like structure.
For instance, general maps can be decomposed into blocks formed by their $2$-connected parts,
which articulate as a tree of blocks.
A multitude of other examples of map families exhibiting such a decomposition scheme
is listed in \cite{BFSS01}.
Let us mention simple triangulations decomposed into irreducible blocks,
bipartite maps into simple bipartite blocks,
but also bicubic maps into $3$-connected blocks \cite{Tutte1963},
or even maps with an additional decoration such as meandric systems decomposed into irreducible blocks \cite{LandoZvonkin92_Meanders}.

The model of block-weighted maps simply amounts to assigning a fixed real positive weight $u$ per block,
and seeing how this weight affects the statistics of the maps at hand.
Remarkably, it was shown in \cite{Fleurat2024,ZSPhD} that a phase transition occurs at a critical value $u_{\mathrm{crit}}$ of the block weight.
For $u < u_{\mathrm{crit}}$, the maps are typically formed of a single macroscopic block
carrying a finite fraction $f(u) > 0$ of the total size $n$ of the map (defined as, say, its number of edges),
with attached outgrowths made of blocks of vanishing density
i.e.\ with a size $o(n)$.
The order parameter $f(u)$ vanishes continuously at $u = u_{\mathrm{crit}}$,
where the largest blocks have size $\Theta(n^{\beta})$
for some $\beta \in [\frac{1}{2},1)$ depending on the universality class of the problem at hand.
For $u > u_{\mathrm{crit}}$, the maps degenerate into trees on a macroscopic scale.

It is worth noting that a similar transition pattern had already been observed in the physics literature,
in the context of Liouville Quantum Gravity (LQG) for models of
random maps attached into tree structures by so-called \emph{touching points}
\cite{Das_1989_critical_matrix_models,Korchemsky_1992_Loops_curvature,Klebanov_1995_Liouville}.
In that context, it was recognized in \cite{DG25} that block-weighted maps indeed provide
a realization of the so-called \emph{LQG duality} which expresses that 
the subcritical phase $u < u_{\mathrm{crit}}$ and the critical phase $u = u_{\mathrm{crit}}$
are linked by a number of universal duality relations.

\bigskip

A nice outcome of the study of \cite{Fleurat2024,ZSPhD}
is that the metric space associated with (undecorated) planar maps endowed with their graph distance
converges, up to scaling depending on $u$, to different continuous metric spaces:
the Brownian sphere for $u < u_{\mathrm{crit}}$,
the $\frac{3}{2}$-stable Lévy tree for $u = u_{\mathrm{crit}}$,
and the Brownian tree for $u > u_{\mathrm{crit}}$.

In this paper, we explore how this change of behavior can be captured by 
one of the simplest metric-sensitive quantities,
the so-called \emph{distance-dependent two-point function} of block-weighted maps,
which, so to say, 
measures at fixed $u$ the distance profile between two points chosen uniformly at random on the map.
More precisely, we focus in this paper on the study of a particular block-weighted map problem considered in \cite{Fleurat2024},
namely that of \emph{planar quadrangulations decomposed into simple blocks},
for which $u_{\mathrm{crit}} = 9/5$.
Some of the metric properties of this decomposition were already discussed in \cite{quadwithnoME}
in the context of so-called \emph{minimal neck baby universes (minbus)}.
In some sense, our study can be viewed as an extension of that paper
where we allow for an arbitrary value of $u$ instead of just considering $u=1$ (general quadrangulations) and $u=0$ (simple quadrangulations).

\bigskip

In addition to the model at hand, our motivations are also
methodological and come from the realm of analytic combinatorics
\cite{Flajolet2009}. Indeed, several models of random planar maps have
the peculiar property that their distance-dependent two-point function
admits an explicit expression. The simplest example is for uniform
random planar quadrangulations: denote by $R_\ell^{(n)}$ the number of
planar quadrangulations with $n$ faces and two marked
points\footnote{More precisely, one vertex and one edge.} at distance
at most $\ell$, then it is known~\cite{geod,onewall} that the
generating function $R_\ell:=\sum_{n=0}^\infty R_\ell^{(n)} g^n$
admits the explicit expression
\begin{equation}
  \label{eq:Rellexplic}
  R_\ell = R \frac{\left(1-x^\ell\right)\left(1-x^{\ell+3}\right)}{\left(1-x^{\ell+1}\right)\left(1-x^{\ell+2}\right)}
\end{equation}
with $R,x$ certain auxiliary series depending on $g$. Similar
expressions exist for other families of random maps or embedded
trees~\cite{DiFrancesco2005,BousquetMelou2006,Kuba2011,AB2013,gen2p}
as well as for other metric-related
quantities~\cite{threepoint,FG2014,AB2016}. Studying the asymptotics
of $R_\ell^{(n)}$ for $n,\ell$ large can be done by extracting the
coefficient of $g^n$ in~\eqref{eq:Rellexplic} using a contour integral
and performing a saddle-point analysis. Note that we have to deal with
an integral in one variable, even though the integrand depends on the
two parameters $n,\ell$. In contrast, simple
quadrangulations~\cite{quadwithnoME}, and a fortiori block-weighted
quadrangulations, do not admit an explicit expression such
as~\eqref{eq:Rellexplic} for their two-point function. Instead, we
have access to this two-point function through a \emph{bivariate}
series which involves also a sum over the distance parameter
$\ell$. As a warm-up, consider the bivariate series
$\hat R(t,g) := \sum_{\ell=1}^\infty R_\ell t^\ell$, which by some
elementary manipulations\footnote{Write
  $R_\ell=R\left(1+\frac{1-x^2}x\left(\frac1{1-x^{\ell+2}}-\frac1{1-x^{\ell+1}}\right)\right)$,
  expand $\frac1{1-x^{\ell+i}}$ as $\sum_{k=0}^{\infty} x^{k(\ell+i)}$
  and interchange the summations over $\ell$ and $k$ in
  $\hat R(t,g)$.}  can be rewritten in the form
\begin{equation}
  \label{eq:merogenform}
  \sum_{k=0}^\infty \frac{a_k}{1-t x^k}
\end{equation}
for some univariate series $a_k$, $k \geq 0$. From the analytic point
of view, this is a meromorphic function in $t$ having poles at
$t=x^{-k}$ for all $k \geq 0$. As visible in
Theorem~\ref{thm:hatKexpr} below, the two-point function of
block-weighted quadrangulations also involves a function of this form.
Such a function does not seem to fall into the class amenable to generic
results of multivariate analytic combinatorics~\cite{ACSV24}. We are
therefore led to develop our own approach, which consists in
performing two successive contour integrals in $t$ and $g$ and
analyzing their asymptotics. A particularly challenging aspect is the
vicinity of the radius of convergence of the series $x$, where we have
$x=1$ so that the poles of~\eqref{eq:merogenform} coalesce, requiring
a delicate analysis. We believe that our approach could be potentially
adapted for other problems in which series similar
to~\eqref{eq:Rellexplic} appear, and as such constitute a valuable
contribution to analytic combinatorics.



\subsection{Basic definitions and notations}

Following the standard terminology, found for example in \cite{Schaeffer15},
we define a \emph{planar map} (hereafter called a map for short)
as a connected (multi)graph drawn on the sphere without edge crossings.
In general, loops and multiple edges are allowed;
a map with no loops nor multiple edges is called \emph{simple}.
A map consists of vertices, edges, faces, and corners.
A map is said \emph{rooted} if it has a distinguished corner,
called the \emph{root corner}, which is incident to the \emph{root vertex};
the edge following the root corner counterclockwise is called the \emph{root edge}.
The \emph{degree} of a face or vertex is its number of incident corners.
A \emph{quadrangulation} is a map whose faces all have degree four.
Note that a planar quadrangulation cannot have loops, however it may have multiple edges.
Therefore, in a simple planar quadrangulation, we just have to forbid multiple edges.

Given a planar quadrangulation $\cQ$, let us cut the sphere along all
the $2$-cycles (cycles of length $2$) of $\cQ$: we obtain a number
$b(\cQ) \geq 1$ of pieces, which we call the \emph{number of blocks}
of $\cQ$. Note that we have $b(\cQ)=1$ if and only if $\cQ$ is
simple. Intuitively speaking, the pieces themselves can be thought as
the ``blocks'' of $\cQ$, but we shall give in
Section~\ref{sec:discreteTPF} a more precise combinatorial definition,
as the blocks are actually simple quadrangulations. Given two
parameters $g,u$, we assign to $\cQ$ the weight
\begin{equation}
  w(\cQ) := g^{f(\cQ)} u^{b(\cQ)-1}
  \label{eq:wQdef}
  \end{equation}
where $f(\cQ)$ denotes the number of faces of $\cQ$. Then, we define
the generating function $Q(g,u)$ of block-weighted planar quadrangulations as
\begin{equation}
  Q(g,u) := \sum_{\cQ} w(\cQ)
  \label{eq:Qgudef}
\end{equation}
where the sum runs over the set of all rooted planar
quadrangulations. Note that, for $u=1$ and $u=0$, this series specializes to
\begin{equation}
  Q(g,1)=\sum_{n \geq 1} \frac{2 \cdot 3^n (2n)!}{n! (n+2)!} g^n, \qquad
  Q(g,0)=\sum_{n \geq 1} \frac{2(3n-3)!}{n!(2n-1)!} g^n
  \label{eq:Qg10}
\end{equation}
which are the generating functions of respectively rooted planar
quadrangulations and rooted simple planar quadrangulations. By a
well-known bijection, these are also the generating functions of
respectively rooted planar maps and rooted non-separable (i.e.\
$2$-connected) planar maps with a weight $g$ per
edge~\cite{Tutte1963}.

Given a real number $u \geq 0$, an integer $n \geq 1$, and a rooted
planar quadrangulation $\cQ$ with $n$ faces, we assign to $\cQ$ a
probability
\begin{equation}
  \Prob(\cQ) := \frac{u^{b(\cQ)-1}}{[g^n] Q(g,u)}
\end{equation}
where $[g^n] Q(g,u)$ is the coefficient of $g^n$ in $Q(g,u)$ (this
coefficient is a polynomial in $u$). Varying $u$ and $n$, we get a
family of probability distributions on quadrangulations, which
constitutes the (fixed size) model of \emph{block-weighted planar
  quadrangulations}.

As discussed in~\cite{Fleurat2024}, this model is closely related to
the model of block-weighted planar maps, in which a block is defined
as a $2$-connected component (in contrast with the simple components
considered for quadrangulations). Our series $Q(g,u)$ is related to
the series $M(g,u)$ of this reference by $M(g,u)=1+u\,Q(g,u)$.

\subsection{Outline of this paper}

The aim of Section~\ref{sec:discreteTPF} is to obtain an expression as explicit as possible for the two-point function of block-weighted quadrangulations.
Section~\ref{ssec:blockweightedquad} provides a precise description of 
the block decomposition of rooted planar quadrangulations into simple blocks,
which results in the fundamental substitution relation of Proposition~\ref{prop:Qsimplesubstitutionrelation} for $Q(g,u)$.
We introduce in Section~\ref{ssec:blockweightedTPF} the two-point function of planar quadrangulations
as the generating function of maps with two marked edges at a prescribed graph distance from each other,
and show that it also satisfies a fundamental substitution relation
which links it to the two-point function of simple quadrangulations,
as computed in \cite{quadwithnoME},
yielding Theorem~\ref{thm:hatKexpr}.
We also introduce the slightly simpler \emph{block two-point function} 
for which the two marked edges are conditioned to belong to the same block (see Theorem~\ref{thm:hatLexpr}).

Section~\ref{sec:asymptQ} is devoted to the large $n$ asymptotic limit
of the ``number'' $[g^n]Q(g,u)$ of maps of size $n$,
as summarized in Proposition~\ref{prop:Qn_asym}.
Even though these asymptotic behaviors can be found in \cite{Fleurat2024},
we present here a full self-contained derivation via a contour integral
which mirrors the transfer theorems of \cite{Flajolet2009},
as it allows us to introduce the various tools that will be needed for the next section.
We introduce in Section~\ref{sec:ratpar} the rational parametrization
of $g$ and $Q$ in the contour integral, that we will use throughout the paper.
We show in Section~\ref{ssec:saddle} how to obtain the large $n$ asymptotics
of $[g^n]Q(g,u)$ from a saddle-point method,
using different explicit integration contours according to the cases: after some preliminaries in Section~\ref{sssec:basicssaddle},
we discuss in Section~\ref{sssec:noncritQ} the subcritical and supercritical cases,
and in Section~\ref{sssec:critQ} the critical case which requires a slightly more involved contour.

Section~\ref{sec:asymptKl} addresses the large $n$ scaling limit of the distance-dependent two-point function itself.
Our main result is summarized in Proposition~\ref{prop:distanceprofiles},
where we identify the appropriate distance scaling
and give fully explicit expressions for the distance profiles in the various phases.
After some general analytic preliminaries in Section~\ref{ssec:asymptKl_prelim},
we detail our calculations for the supercritical regime in Section~\ref{ssec:supercritKl}.
The study of the other regimes requires the analysis of the singularity of an auxiliary function,
which is performed in Section~\ref{ssec:hsingtext} and Appendix~\ref{app:hsing}.
We then perform the calculations for the critical and subcritical regimes
in Sections~\ref{ssec:critKl} and~\ref{ssec:subcritKl} respectively.
We finally compute in Section~\ref{ssec:rootblockcrit}
the scaling limit of the block two-point function at the critical point 
and the associated distance profile.

We conclude in Section~\ref{sec:conc} by discussing the compatibility of our results
with those of \cite{Fleurat2024}.

\paragraph{Acknowledgments.} We thank Guillaume Chapuy and Gr\'egory Miermont for useful discussions. This work is partially supported by the ANR  grant CartesEtPlus ANR-23-CE48-0018.

\section{The discrete two-point function}
\label{sec:discreteTPF}

Throughout this section, we are dealing with generating functions of
quadrangulations, which we treat as formal power series.

\subsection{The block decomposition of rooted quadrangulations}
\label{ssec:blockweightedquad}

Let us give a precise combinatorial definition of the notion of block of a quadrangulation.
We first introduce some terminology.
Given a rooted planar quadrangulation $\cQ$, let us consider a $2$-cycle $\mathcal C$.
The \emph{interior} of $\mathcal C$ is the closed region delimited by $\mathcal C$ not containing the root corner.
We say that $\mathcal C$ is:
\begin{itemize}
  \item \emph{maximal} if it is not contained in the interior of another $2$-cycle,
  \item a \emph{neck} if, in the interior of $\mathcal C$, there is no edge connecting the two vertices visited by $\mathcal C$, apart from those from $\mathcal C$.
\end{itemize}
It is easily seen that the interiors of two different maximal $2$-cycles cannot share a common edge.

The \emph{subquadrangulation} $\cQ_\mathcal C$ associated with $\mathcal C$ is defined as the quadrangulation obtained by considering the interior of $\mathcal C$ and identifying its two outer edges together into a single edge denoted $e$.
We canonically root $\cQ_\mathcal C$ as follows:
consider the vertex $v$ of $\mathcal C$ closest\footnote{In \cite[Definition 7]{Fleurat2024}, the rooting is obtained in the general maps associated to the quadrangulations, which instead corresponds to choosing the vertex at \emph{even distance} from the root. Our convention will conserve the distances (see Figure~\ref{fig:recursionk_g}), which was not necessary for \cite{Fleurat2024}.} to the root of $\cQ$;
then we root $\cQ_\mathcal C$ at the corner following $e$ clockwise around $v$.

Let us now describe the block decomposition of $\cQ$.
We first define the \emph{root block} of $\cQ$ as the rooted simple quadrangulation obtained by 
``squeezing'' the interior of each maximal $2$-cycle,
i.e.\ removing its interior and merging its two sides into a single edge.
The recursive block decomposition of $\cQ$ is done as follows:
\begin{itemize}
  \item if $\cQ$ is simple, then it is identical to its root block, and we do nothing;
  \item otherwise, we split $\cQ$ into its root block
    and into the subquadrangulations associated with all its maximal $2$-cycles,
    which we recursively decompose.
\end{itemize}
A \emph{block} of $\cQ$ is defined as the root block of a quadrangulation encountered during this recursive decomposition.
It is called an \emph{inner block} of $\cQ$ if it is not equal to the root block of $\cQ$.
Note that the factor $u^{b(\cQ) -1}$ in \eqref{eq:wQdef} corresponds to a weight $u$ per inner block of $\cQ$.
Note furthermore that each face of $\cQ$ corresponds to a face in exactly one (root or inner) block of $\cQ$.

Interestingly, we may also give a global, non-recursive description of a block.
Indeed, it may be checked that there is a bijection between
the set of inner blocks of $\cQ$ and the set of its necks.
More precisely, given a neck $\mathcal C$,
the root block of the subquadrangulation $\cQ_\mathcal C$ is an inner block of $\cQ$,
and all the inner blocks of $\cQ$ are obtained in this way.

As just mentioned, in the above block decomposition, each face of $\cQ$ belongs to exactly one block.
Each edge can also be canonically assigned to exactly one block:
by convention, we decide that any given edge of $\cQ$ belongs to the same block
as the incident face \emph{on its right}, 
where we canonically orient each edge
by going away from the root vertex.
With that convention, the root edge belongs to the root block.

Armed with the above definitions, we may now establish the following proposition.
\begin{prop}
  \label{prop:Qsimplesubstitutionrelation}
  Denoting $\Qs(g):=Q(g,0)$ the univariate generating
  function of rooted simple quadrangulations, we have the following substitution relation:
\begin{equation}
  Q(g,u) = \Qs\left(z(g,u)\right), \qquad z(g,u):=g\left(1+u\,Q(g,u)\right)^2.
  \label{eq:Qsimplesubstitutionrelation}
\end{equation}
\end{prop}

\begin{proof}
  Define a \emph{decorated simple quadrangulation} as a rooted simple planar quadrangulation
  where each edge is possibly decorated by a rooted, not necessarily simple, planar quadrangulation.
  We claim that there is a bijection between rooted planar quadrangulations
  and decorated simple quadrangulations.
  The direct mapping is precisely the recursion step of the block decomposition,
  where the subquadrangulations are precisely the decorations.
  Conversely, given a decorated simple quadrangulation,
  we reconstruct a general planar quadrangulation
  by inflating each decorated edge into a $2$-cycle,
  and gluing inside it the decoration where we similarly inflated the root edge\footnote{
  We match the root vertex of the decoration with the vertex incident to the decorated edge
  closest to the root vertex of the simple quadrangulation.}.
  
  Under this bijection, giving the weight \eqref{eq:wQdef} to a rooted planar quadrangulation $\cQ$
  amounts in the associated decorated simple quadrangulation to
  assigning a weight $g$ per face of the simple quadrangulation,
  a weight $u$ per decorated edge,
  and multiplying by the weights of all decorations.
  This is equivalent to just attaching a weight $\nobreak{z(g,u)=g(1+u\, Q(g,u))^2}$,
  to each face of a simple quadrangulation, since there are twice as many edges as faces,
  and each edge may either be decorated (thereby contributing a weight $u\, Q(g,u)$)
  or not (thereby contributing $1$).
  This leads directly to equation \eqref{eq:Qsimplesubstitutionrelation}.
\end{proof}

\subsection{Block-weighted two-point functions}
\label{ssec:blockweightedTPF}

We are now interested in enumerating block-weighted rooted quadrangulations with an \emph{additional marked edge}.
We call \emph{root distance} of the marked edge the graph distance from the root vertex to the furthest vertex incident to that marked edge.
A marked edge at root distance $\ell$ ($\ell\geq 1$) then connects a vertex at \emph{root distance} $(\ell-1)$ to one at root distance $\ell$.

Let us denote by $K_\ell(g,u)$ the generating function of rooted quadrangulations
with an additional marked edge at root distance $\ell$,
with a weight as in \eqref{eq:wQdef}.
For $\ell = 1$, we conventionally require that the root edge and the marked edge do not have the same two endpoints,
as it makes our formulas simpler.
We further introduce the series 
\begin{equation}
  \hat K(t,g,u) := \sum_{\ell\geq 1}t^{\ell-1} K_\ell(g,u).
  \label{eq:defhatK}
\end{equation}
The quantity $K_\ell(g,u)$, or equivalently $\hat K(t,g,u)$, 
is called the \emph{distance-dependent two-point function} of block-weighted planar quadrangulations.

By refining the arguments leading to Proposition~\ref{prop:Qsimplesubstitutionrelation},
we obtain the following:
\begin{prop}
  \label{prop:hatKsubst}
  Denoting by $\hatKs(t,g):= \hat K(t,g,0)$ the distance-dependent two-point function of simple planar quadrangulations, we have the relation
  \begin{equation}
    \hat K(t,g,u) = (1 + u\, Q(g,u))^2 \frac{\hatKs(t,z(g,u))}{1 - u \, \hatKs (t,z(g,u))}, \qquad z(g,u):=g\left(1+u\,Q(g,u)\right)^2.
    \label{eq:substhatK}
  \end{equation}
\end{prop}

\begin{proof}
  We say that an edge is \emph{simple} if it is not part of a $2$-cycle.
  We call $k_\ell(z, u)$ the generating function of rooted quadrangulations with an additional marked edge at root distance $\ell$, different from the root edge 
  and such that the root edge and the marked edge are simple edges
  and \emph{all the $2$-cycles contain the marked edge in their interior} (see Figure~\ref{fig:recursionk_g} for an illustration).
  A quadrangulation is counted with a weight $z$ per face and $u$ per inner block.
  
  We claim that we have the relation
  \begin{equation}
    k_\ell(z, u) = \Ks_\ell(z) + u\sum_{m=1}^\ell \Ks_m(z) k_{\ell+1-m}(z,u), \qquad \ell\geq 1
    \label{eq:recursionk_g}
  \end{equation}
  \begin{figure}
    \centering
    \includegraphics{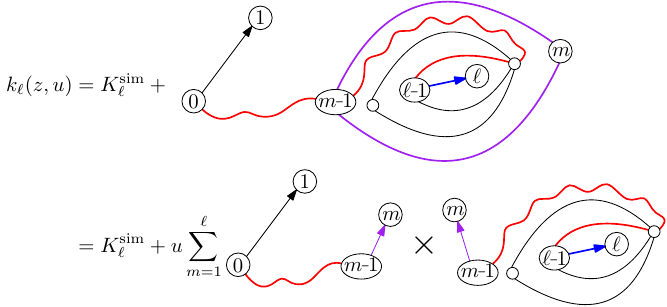}
    \caption{Graphical proof of Equation~\eqref{eq:recursionk_g}.
      The labels represent the root distances.
      All the $2$-cycles contain the blue marked edge in their interior.
      The unique maximal $2$-cycle $\mathcal C$ is drawn in purple on the top figure. 
      The red path is one of the shortest paths from the root vertex to the marked edge,
      one of those that pass through the vertex represented with label $m-1$.
      Squeezing $\mathcal C$ yields the simple quadrangulation displayed on the bottom left,
      with a marking on the resulting edge at root distance $m$.
      The subquadrangulation $\mathcal Q_{\mathcal C}$ is represented on the bottom right, rooted at the purple edge,
      but there the labels (which we did not change) are the root distances shifted by $m-1$.
    }
    \label{fig:recursionk_g}
  \end{figure}
where $\Ks_\ell(g):=K_\ell(g,0)$ is the generating function of \emph{simple} rooted quadrangulations with an additional marked edge at root distance $\ell$, different from the root edge.
  Indeed, consider a quadrangulation $\cQ$ enumerated by $k_\ell(z,u)$.
  If it has no $2$-cycle, then it is counted by $\Ks_\ell(z)$ with the correct weight.
  Else, there is a unique maximal $2$-cycle $\mathcal C$,
  as otherwise we would obtain two maximal $2$-cycles whose interiors share a common edge, since
  they both must include the marked edge.
  Moreover, this unique maximal $2$-cycle $\mathcal C$ is necessarily a neck:
  supposing there is a third edge connecting the two vertices visited by $\mathcal C$,
  this edge cannot be equal to the marked edge which is assumed to be simple,
  hence it splits the interior of $\mathcal C$ into two domains delimited by $2$-cycles,
  one of which not containing the marked edge, which is not allowed.
  Squeezing the interior of the neck $\mathcal C$, we obtain a new marked edge in the root block of $\cQ$.
  Calling $m$ the root distance of that new marked edge,
  the root block is a quadrangulation enumerated by $\Ks_m(z)$.
  As for $\cQ_\mathcal C$, first remark that the root distance of the marked edge within $\cQ_\mathcal C$ is $\ell+1-m$,
  since we can find a geodesic in $\cQ$ from the root vertex to the marked edge that enters $\cQ_\mathcal C$ at its root vertex and stays in $\cQ_\mathcal C$ afterwards, see Figure~\ref{fig:recursionk_g} for an illustration.
  Also, since $\mathcal C$ is a neck, the root edge of $\cQ_\mathcal C$ is simple,
  so $\cQ_\mathcal C$ is enumerated by $k_{\ell+1-m}(z,u)$.
  The inner blocks of $\cQ$ are all the blocks of $\cQ_\mathcal C$,
  hence the additional weight $u$ in \eqref{eq:recursionk_g}
  to account for the root block of $\cQ_\mathcal C$.
  Note that for a quadrangulation counted by $k_\ell(z,u)$,
  its number of inner blocks is identical to its number of $2$-cycles, which are all necks.

  We rewrite the relation~\eqref{eq:recursionk_g} in the form
  \begin{equation}
    \hat k(t,z,u) = \left(1+u \hat k(t,z,u)\right) \hatKs(t,z)
    \label{eq:reckhat}
  \end{equation}
  by introducing the series
  \begin{equation}
    \hat k(t,z,u):=\sum\limits_{\ell\geq 1}t^{\ell-1} k_\ell(z,u)
    \label{eq:ktzudef}
  \end{equation}
  and noting that $\hatKs(t,z):=\sum\limits_{\ell\geq 1}t^{\ell-1}\Ks_\ell(z)$.
  The relation~\eqref{eq:reckhat} amounts to
  \begin{equation}
    \hat k(t,z,u) = \frac{\hatKs(t,z)}{1 - u\, \hatKs(t,z)}.
    \label{eq:hatk_to_hatKs}
  \end{equation}

  We now prove by substitution that
  \begin{equation}
    K_\ell(g,u) = (1 + u\, Q(g,u))^2 k_\ell(z(g,u),u), \qquad z(g,u):=g\left(1+u\,Q(g,u)\right)^2.
    \label{eq:substKl}
  \end{equation}
  We shall proceed similarly to Subsection~\ref{ssec:blockweightedquad}, but we only squeeze $2$-cycles which do not separate the root edge from the marked edge.
  More precisely, let us consider a quadrangulation $\cQ$ enumerated by $K_\ell$,
  and the map $\mathcal K$ obtained by inflating the root edge and the marked edge into $2$-gons.
  We say that a $2$-cycle $\mathcal C$ of $\mathcal K$ is \emph{non-separating}
  if the two $2$-gons lie on the same side of $\mathcal C$;
  the side containing none of the $2$-gons is called the \emph{inside} of $\mathcal C$.
  A non-separating $2$-cycle $\mathcal C$ is \emph{maximally non-separating}
  if it is not contained in the inside of another non-separating $2$-cycle.
  Squeezing the inside of all the maximally non-separating $2$-cycles of $\mathcal K$ and the two $2$-gons yields
  a quadrangulation $\mathfrak q$ with the same root and marked edges, enumerated by $k_\ell$.
  Each edge $e$ of $\mathfrak q$ different from the root and the marked edge
  corresponds to either an original edge of $\cQ$,
  or to a maximally non-separating $2$-cycle $\mathcal C_e$ whose inside is an arbitrary quadrangulation,
  which we use as a decoration of $e$ in $\mathfrak q$.
  This leads us to assign the weight $(1+ u\, Q(g,u))$ to each edge $e$ of $\mathfrak q$ different from the root and the marked edge.
  As for the marked edge and the root edge, instead of one (possibly empty) decoration,
  we must associate to each of them \emph{two} (possibly empty) decorations, one on each side of that edge.
  This yields in the generating function $K_\ell$ a total of four $(1 + u\, Q(g,u))$ factors instead of only two.
  Altogether, we are thus led to assigning the weight $g$ per face and $(1+ u\, Q(g,u))$ per edge,
  or equivalently the weight $z(g,u)=g\left(1+u\,Q(g,u)\right)^2$ per face of $\mathfrak q$,
  and to adding a global prefactor $(1+ u\, Q(g,u))^2$, which yields equation \eqref{eq:substKl}.

  The wanted equation \eqref{eq:substhatK} comes from summing \eqref{eq:substKl} over $\ell$ and using \eqref{eq:hatk_to_hatKs}. 
\end{proof}

The interest of Proposition~\ref{prop:hatKsubst} is that, combined
with a result from~\cite{quadwithnoME}, it gives an explicit
expression for $\hat K(t,g,u)$:

\begin{thm}
  \label{thm:hatKexpr}
  The distance-dependent two-point function of block-weighted planar
  quadrangulations is given by
  \begin{equation}
    \hat K(t,g,u) = \frac{z(g,u)}g \cdot \frac{\hat h(t,z(g,u))-1}{u + (1-u) \hat h (t, z(g,u))}, \qquad z(g,u):=g\left(1+u\,Q(g,u)\right)^2
    \label{eq:hatKexpr}
  \end{equation}
  with
  \begin{equation}
    \hat h(t,z) = r^2 \frac{(1-x^2)^2}{x^2} \sum_{k \geq 1} k x^{2k} \frac{1-x^{k}}{1-t x^{k}}
    \label{eq:hathexpr}
  \end{equation}
  where $r,x$ are the unique power series in $z$ satisfying
  \begin{equation}
    r = 1+z\, r^3, \qquad x+\frac 1 x + 1 = \frac 1 {z\, r^2}.
    \label{eq:rxzeqs}
  \end{equation}
\end{thm}

\begin{proof}
  From \cite[Equation (2.29)]{quadwithnoME}, we know that the series\footnote{The definition of $h_\ell(z)$ in \cite{quadwithnoME} includes an additional term $1$ for $\ell=1$,
  which makes the expression \eqref{eq:klFromMinbus} valid for all $\ell \geq 1$.}
$h_\ell(z) = k_\ell(z,1) + \delta_{\ell,1}$ admits the explicit expression
\begin{equation}
  h_\ell(z) = r^2 x^{\ell -1}\frac{(1-x)(1-x^2)^2(1-x^{2\ell+3})}{(1-x^{\ell+1})^2(1-x^{\ell+2})^2}
  \label{eq:klFromMinbus}
\end{equation}
with $r,x$ determined by~\eqref{eq:rxzeqs}.
Introducing the series $\hat h(t,z) := \sum\limits_{\ell\geq 1}t^{\ell-1}h_\ell(z)$ and 
using the partial fraction decomposition
\begin{equation}
  x^{\ell -1}\frac{(1-x)(1-x^{2\ell+3})}{(1-x^{\ell+1})^2(1-x^{\ell+2})^2}
  =\frac{1}{x^2}  \left( \frac{x^{\ell+1}}{(1-x^{\ell+1})^2} - \frac{x^{\ell+2}}{(1-x^{\ell+2})^2} \right),
  \label{eq:kl_simpleelements}
\end{equation}
we find from~\eqref{eq:klFromMinbus} that $\hat h(t,z)$ admits the expressions
\begin{equation}
  \begin{split}
    \hat{h}(t,z) &= r^2 \frac{(1-x^2)^2}{x^2} \sum_{\ell \geq 1} \left( \frac{x^{\ell+1}}{(1-x^{\ell+1})^2} - \frac{x^{\ell+2}}{(1-x^{\ell+2})^2} \right) t^{\ell-1} \\
                 &= r^2 \frac{(1-x^2)^2}{x^2} \sum_{\ell \geq 1} \sum_{k \geq 1} k \left( x^{k(\ell+1)} - x^{k(\ell+2)} \right) t^{\ell-1} \\
                 &= r^2 \frac{(1-x^2)^2}{x^2} \sum_{k \geq 1} k x^{2k} (1-x^k) \sum_{\ell \geq 1} (t x^{k})^{\ell-1} \\
                 &= r^2 \frac{(1-x^2)^2}{x^2} \sum_{k \geq 1} k x^{2k} \frac{1-x^{k}}{1-t x^{k}}
  \end{split}
  \label{eq:hhatmero}
\end{equation}
leading to \eqref{eq:hathexpr}. Note that all these computations make sense in $\C[[t,z]]$, since $x$ is a series in $z$ with no constant term.
Now, inverting~\eqref{eq:hatk_to_hatKs} at $u=1$ and noting that
$\hat h(t,z) = \hat k(t,z,1) + 1$, we may write
\begin{equation}
  \hatKs (t,z)=
  1 - \frac1{\hat h(t,z)}.
  \label{eq:hatKs_to_hath_u1}
\end{equation}
Plugging this expression into~\eqref{eq:substhatK}, the wanted result follows.
\end{proof}

\begin{rem}
  \label{rem:hath_t_eq_1}
  For later use, we record that, at $t=1$, \eqref{eq:hathexpr} simply evaluates to
  \begin{equation}
    \hat{h}(1,z) = r^2,
    \label{eq:hath_t_eq_1}
  \end{equation}
  which is also visible from the straightforward alternative expression:
  \begin{equation}
    \hat{h}(t,z) = r^2 \left( 1-(1-t)\frac{(1-x^2)^2}{x^2} \sum_{k \geq 1} k\frac{x^{3k}}{1-t x^{k}} \right).
    \label{eq:hhatmeroalt}
  \end{equation}
  By~\eqref{eq:hatKexpr}, we deduce that
  \begin{equation}
    \hat K(1,g,u) = \frac{z(g,u)}g \cdot \frac{r^2-1}{u + (1-u) r^2}
  \end{equation}
  with $r$ taken at $z=z(g,u)$.
  We may check that this expression is consistent with
  \begin{equation}
    \hat K(1,g,u) = 2g \frac{\partial Q}{\partial g}(g,u) - Q(g,u) - u \, Q(g,u)^2
    \label{eq:hatK1_combi}
  \end{equation}
  which follows from the combinatorial definition of $\hat K(1,g,u)$:
  it is the generating function of rooted quadrangulations with an additional marked edge,
  that is neither equal to the root edge
  nor has the same two endpoints.
\end{rem}

\begin{rem}
  The relations~\eqref{eq:reckhat} at $u=1$, and
  \eqref{eq:hatKs_to_hath_u1}, correspond precisely
  to~\cite[Equation~(2.35)]{quadwithnoME} where the series
  $\hat g(t,z)$ is the same as the series $\hatKs (t,z)$ here.
\end{rem}

To conclude this section, let us discuss an interesting variant of
the two-point function, namely the \emph{block two-point function} defined as
the generating function $L_\ell(g,u)$ of those quadrangulations $\cQ$
in the set enumerated by $K_\ell(g,u)$ whose marked edge lies in the root block of $\cQ$.

By arguments similar to the above discussion, we have the following relation:
\begin{equation}
  L_\ell(g,u) = \Ks_\ell (z(g,u))
  \label{eq:substLl}
\end{equation}
with $z(g,u)$ as before.
Note, when comparing with \eqref{eq:substKl}, the absence of the prefactor $(1+ u\, Q(g,u))^2$. 
Indeed, demanding that the marked edge belongs to the root block
forbids adding decorations on the right of the root edge or of the marked edge,
so that both edges belong to the same block.
Forbidding a decoration on the right of the root edge ensures that
the quadrangulation $\mathfrak q$ in the above discussion is the root block;
forbidding a decoration on the right of the marked edge ensures that
this edge belongs to $\mathfrak q$.

Using \eqref{eq:hatKs_to_hath_u1} and \eqref{eq:substLl}, we get the following counterpart of Theorem~\ref{thm:hatKexpr}:
\begin{thm}
  \label{thm:hatLexpr}
  With the above notations, 
  the distance-dependent block two-point function
  $\hat L(t,g,u) := \sum\limits_{\ell\geq 1}t^{\ell-1} L_\ell(g,u)$
reads explicitly
  \begin{equation}
    \hat L(t,g,u) = 1- \frac{1}{\hat h(t,z(g,u))}, \qquad z(g,u):=g\left(1+u\,Q(g,u)\right)^2.
    \label{eq:substhatL}
  \end{equation}
\end{thm}

\begin{rem}
  Let us note that we have
  \begin{equation}
    \hat L(1,g,0) = \hatKs (1,g) = \left(2g \frac \partial {\partial g} - 1\right) \Qs(g).
    \label{eq:edgeMarkingSimple}
  \end{equation}
  Indeed, if $u=0$ then all our quadrangulations are simple.
  If by setting $t=1$ we do not restrict the root distance of the marked edge,
  then the quadrangulations enumerated by $\hatKs (1,g)$ are nothing but
  the simple quadrangulations with a marked edge different from the root edge,
  in number $2n-1$ if the quadrangulation has $n$ faces.
\end{rem}

\section{Asymptotic enumeration}
\label{sec:asymptQ}

From now on, we fix the block weight $u$ to a nonnegative real value,
and we denote by $Q_n$ the coefficient of $g^n$ in $Q(g,u)$.  Here and
in several notations below, we leave the dependency in $u$ implicit
for brevity. Note that $Q_n$ is a nonnegative real number, since
$Q(g,u)$ is from its very definition~\eqref{eq:Qgudef} a series in $g$
whose coefficients are polynomials in $u$ with nonnegative integer coefficients.
Our purpose is to work out
the asymptotics of $Q_n$ by applying the saddle-point method in a
fully justified manner. Even though the results here could also be
obtained by a routine application of so-called sim-transfer
\cite[Corollary VI.1]{Flajolet2009}, we will later go outside the
domain of validity of this result, and our bottom-up approach will
prove more versatile.
Precisely, we will prove in this section the following result, already found in \cite{ZSPhD}:
\begin{prop}
  \label{prop:Qn_asym}
  The asymptotic behavior of $Q_n$ for $n \to \infty$ is as follows:
  \begin{itemize}
  \item in the subcritical phase $0 \leq u < 9/5$, we have
    \begin{equation} \label{eq:Qn_asymsub}
      Q_n \sim \frac{2}{\sqrt{\pi}} \left( \frac{3+u}{9-5u} \right)^{5/2} \left( \frac{3(3+u)^2}4 \right)^n n^{-5/2},
    \end{equation}
  \item at the critical point $u=9/5$, we have
    \begin{equation} \label{eq:Qn_asymcrit}
      Q_n \sim \frac{2^{7/3}}{9\, \Gamma(1/3)} \left(\frac{432}{25}\right)^n n^{-5/3},
    \end{equation}
  \item in the supercritical phase $u>9/5$, we have:
    \begin{equation} \label{eq:Qn_asymsuper}
      Q_n \sim \frac 2 9 \ys \sqrt{\frac3\pi(1-\ys)(3-\ys)} \left( \frac{432}{\ys(6-\ys)^2} \right)^n n^{-3/2}
    \end{equation}
    where $\ys =  3\left(1- \sqrt{\frac{u-1}u}\right)$.
  \end{itemize}
\end{prop}

Our starting point to establish this proposition is the contour integral representation
\begin{equation}
  Q_n = \frac{1}{2i\pi} \oint \frac{Q(g,u) dg}{g^{n+1}}
  \label{eq:Qnintbasic}
\end{equation}
where the contour of integration is initially a small circle around
the origin. Incidentally, note that this requires viewing $Q(g,u)$ as
an analytic function of $g$, instead of a formal power series as was
done in the previous section. It is well-known and apparent
from~\eqref{eq:Qg10} that $Q(g,1)$ has radius of convergence
$\frac{1}{12}$. Hence, $Q(g,u)$ has radius of convergence at least
$\frac1{12}$ for $u\leq 1$, and at least $\frac1{12u}$ for $u \geq 1$
by the naive bound $b(\cQ) \leq f(\cQ)$ holding for any rooted planar
quadrangulation $\cQ$. So, fortunately, the above contour integral makes sense for
any $u \geq 0$.

\subsection{Rewriting the contour integral via rational parametrization}
\label{sec:ratpar}

To analyze the large $n$ behavior of the above integral, it is first
convenient to perform a change of variable, using the following
rational parametrization:

\begin{prop}
  \label{prop:Qratparam}
  Consider the rational function
  \begin{equation}
    g(y):= \frac{y(3-y)^2}{3
      \left(3 + u y(2-y)\right)^2}.
    \label{eq:gydef}
  \end{equation}
  Then, the generating function $Q(g,u)$ of block-weighted planar quadrangulations 
  satisfies
  \begin{equation} \label{eq:Qgyu}
    Q(g(y),u) = \frac{y(2-y)}{3}.
  \end{equation}
\end{prop}

\begin{proof}
  From~\cite[Proposition~2]{BFSS01} and the bijection between rooted
  planar non-separable maps and rooted planar simple quadrangulations,
  it is known that the series $\Qs(z)$ is given by
  \begin{equation}
    \Qs(z) = \frac{y(z)(2-y(z))}{3}
    \label{eq:Qsparam}
  \end{equation}
  where $y(z)$ is the unique series in $z$ such that
  \begin{equation}
    z = \frac{y(z)(3-y(z))^2}{27}, \qquad y(0)=0.
    \label{eq:zparam}
  \end{equation}
  Let us replace in these relations $z$ by
  $z(g,u):=g \left(1+u\,Q(g,u)\right)^2$. Note that $\Qs(z(g,u))$ and
  $y(z(g,u))$ are well-defined bivariate series, since $z(g,u)$ has
  no constant coefficient. Writing $y$ as a shorthand
  notation for $y(z(g,u))$, \eqref{eq:zparam} yields
    $z(g,u) = \frac{y(3-y)^2}{27}$
  while \eqref{eq:Qsimplesubstitutionrelation} and \eqref{eq:Qsparam}
  yield $Q(g,u) = \frac{y(2-y)}{3}$. From these two relations we may
  write
  \begin{equation}
    g = \frac{z(g,u)}{\left(1+u\, Q(g,u)\right)^{2}}=\frac{\frac{y(3-y)^2}{27}}{
      \left(1 + u \frac{y(2-y)}3\right)^2} = g(y).
  \end{equation}
  Since $y$ is a series in $g$ with zero constant coefficient,
  the equation $g=g(y)$ uniquely determines it.
  We get the statement of the proposition by reversely viewing
  $g=g(y)$ as a series in the variable $y$, which we substitute in
  $Q(g,u)$.
\end{proof}

\begin{rem}
  \label{rem:zy_param}
  For later use, we note that $z(g,u)= g \left(1+u\,Q(g,u)\right)^2$ admits, by the above arguments,
  the rational parametrization
  \begin{equation}
    z(g(y),u)=z(y), \qquad z(y):=\frac{y(3-y)^2}{27}
    \label{eq:zgy}
  \end{equation}
  which is independent of $u$.
\end{rem}

Proposition~\ref{prop:Qratparam} allows us to perform the change of
variable $g=g(y)$ in the contour integral
representation~\eqref{eq:Qnintbasic} of $Q_n(u)$, to yield
\begin{equation}
  Q_n = \frac{1}{2i\pi} \oint \frac{y(2-y)}{3}  \frac{g'(y) dy}{g(y)^{n+1}}.  
  \label{eq:Qnintylong}
\end{equation}
Assuming $n>0$, we integrate by parts to get the simpler expression
\begin{equation}
  Q_n = \frac{1}{3i\pi n} \oint \frac{1-y}{g(y)^n} dy.
  \label{eq:Qnintysimp}
\end{equation}
Again, the initial contour of integration is a small circle around the
origin, but since the integrand is a rational function with poles only
at $y=0,3$, we are free to deform the integration contour as long as
we do not cross these poles.

\subsection{Saddle-point approximation}
\label{ssec:saddle}
\subsubsection{Preliminaries}
\label{sssec:basicssaddle}

We will now evaluate the large $n$ behavior of $Q_n$ by applying the
saddle-point method to the right-hand side of
equation~\eqref{eq:Qnintysimp}.  Generally speaking, the method
consists in deforming the integration contour so that it passes
through a \emph{saddle point} of $g$, that is a zero of $g'$. When the
deformation is done appropriately, it is possible to show that the
large $n$ behavior of the integral is dominated by the vicinity of
the saddle point, whose contribution can be computed explicitly.

Here, the saddle points of $g$ are the roots of
\begin{equation}
  g'(y)=\frac{(3-y)(1-y)(uy^2-6uy+9)}{3\left(3 + u y(2-y)\right)^3}.
  \label{eq:gprime_y_eq_0}
\end{equation}
There are four such roots, namely $y=1$,
$y=3$ and $y=3\left(1\pm \sqrt{\frac{u-1}u}\right)$. The smallest real
root, which we will show to be the relevant one, is
\begin{equation}
  \ys := \begin{cases}
    1 & \text{for $u \leq 9/5$,}\\
    3\left(1- \sqrt{\frac{u-1}u}\right) & \text{for $u \geq 9/5$.}\\
  \end{cases}
  \label{eq:ysaddle}
\end{equation}
Note that the two determinations coincide at the critical point $u=9/5$.

\begin{rem} \label{rem:uysinv} In the supercritical phase $u>9/5$, $u$ may be
  conversely expressed in terms of $\ys$ as
  \begin{equation}
    u=\frac9{\ys(6-\ys)}
    \label{eq:uysinv}
  \end{equation}
  which decreases from $+\infty$ to $9/5$ as $\ys$ increases from $0$
  to $1$.  This relation will be used below to recast expressions
  depending on $u$ and $\ys$ as functions of $\ys$ only.
\end{rem}

Let us observe that
\begin{equation}
  g\left(\ys\right) = \begin{cases}
    \displaystyle\frac{4}{3(3+u)^2} & \text{for $u < 9/5$,}\\
    \displaystyle\frac{25}{432} & \text{for $u = 9/5$,}\\
    \displaystyle\frac{\ys\left(6-\ys\right)^2}{432} & \text{for $u > 9/5$}\\
  \end{cases}
  \label{eq:gysaddle}
\end{equation}
where the last line follows from a first application of
Remark~\ref{rem:uysinv}. Anticipating that the
integral~\eqref{eq:Qnintysimp} is indeed dominated by the vicinity of
$\ys$, we expect to have
\begin{equation}
  \lim_{n \to \infty} Q_n^{1/n} = g\left(\ys\right)^{-1}.
\end{equation}
This explains the exponential factors in Proposition~\ref{prop:Qn_asym}.

\subsubsection{The subcritical and supercritical cases}
\label{sssec:noncritQ}

The purpose of this subsection is to establish
Proposition~\ref{prop:Qn_asym} for $u \neq 9/5$, that is when $\ys$ is
a simple root of $g'$. This corresponds to the most generic situation
for saddle-point approximation:

\begin{prop}
  \label{prop:genericsaddle}
  Consider an integral of the form
  \begin{equation}
    I_n = \oint \frac{f(y)}{g(y)^n} dy
    \label{eq:genIn}
  \end{equation}
  taken over a small circle around the origin, with $f$ analytic over
  a domain containing the closed disk $\{|y|\leq y_s\}$. If
  $u \neq 9/5$ and $f(\ys)\neq 0$ then we have for $n \to \infty$
  \begin{equation}
    I_n \sim i f(\ys) \sqrt{\frac{2\pi g(\ys)}{-g''(\ys)}}\, g(\ys)^{-n}\,  n^{-1/2}.
    \label{eq:genInasy}
  \end{equation}
\end{prop}

Before establishing this proposition, let us first apply it to the
asymptotics of $Q_n$. The supercritical case $u>9/5$ is
straightforward: in view of~\eqref{eq:Qnintysimp} we just need to
apply the proposition with $f(y)=1-y$. Using Remark~\ref{rem:uysinv}
we find
\begin{equation}
  g''(\ys) = - \frac{(1-\ys)(6-\ys)^2}{288\ys(3-\ys)}
  \label{eq:gsecond_ys_supercrit}
\end{equation}
which, together with~\eqref{eq:gysaddle}, leads directly to the wanted
asymptotics~\eqref{eq:Qn_asymsuper}.

The subcritical case $u<9/5$ adds a slight difficulty, since the
integrand in~\eqref{eq:Qnintysimp} vanishes at the saddle point
$\ys=1$. We may circumvent this issue by performing an extra
integration by parts, to yield
\begin{equation}
  Q_n = \frac{1}{3 i \pi n (n-1)} \oint \frac{f(y)}{g(y)^{n-1}}dy, \qquad f(y) := \frac{d}{dy} \left ( \frac{1-y}{g'(y)} \right ).
  \label{eq:Qn_doubleintbyparts}
\end{equation}
We then get the wanted asymptotics~\eqref{eq:Qn_asymsub} using the
explicit values~\eqref{eq:gysaddle} and
\begin{equation}
  f(\ys)=f(1)=\left ( \frac{3(3+u)^2}{2(9-5u)} \right )^2, \qquad
  g''(\ys) = g''(1) = - \frac{2(9-5u)}{3(3+u)^3}.
  \label{eq:values_subcrit}
\end{equation}

Let us now establish Proposition~\ref{prop:genericsaddle}. As already
mentioned, the saddle-point method consists in deforming the
integration contour in $I_n$ so that it passes through a saddle point
of $g$. A natural choice for such a contour is the circle of radius
$\ys$ centered at the origin (recall that $\ys$ is the smallest real
saddle point), see Figure~\ref{fig:circleArcs}. The following lemma ensures that the vicinity of $\ys$
indeed dominates the large $n$ behavior of $I_n$.

\begin{lem}
  \label{lem:circlegood}
  For any $u\geq0$, as $y$ varies over the circle of radius $\ys$
  centered at the origin, the function $y \mapsto |g(y)|$ admits a
  strict global minimum at $y=\ys$, a strict global maximum at
  $y=-\ys$, and varies monotonically over the two half-circles
  connecting these extrema.
\end{lem}

\begin{proof}
  It is convenient to use the rational parametrization of the circle
  \begin{equation}
    y(t) = \ys \frac{(1+it)^2}{1+t^2},\ t \in \R
  \end{equation}
  which is such that $y(0)=\ys$ and
  $\lim_{t \to \pm \infty} y(t)=-\ys$.

  As $y(-t)$ is the complex conjugate of $y(t)$ for $t$ real, the
  squared modulus of $g$ is simply given by $g(y(t))g(y(-t))$, which
  turns out to be the square of a rational function
  $\mathrm{Rat}(t^2)$ in $t^2$ with polynomial coefficients in $u$ and
  $\ys$ (which itself depends on $u$). We do not give its explicit
  expression here, as all that we write is best checked on a computer
  algebra system. We want to prove that $\mathrm{Rat}(x)$, where we write $x:=t^2$, is strictly
  increasing for $x \in \R_+$, which we do by computing its
  derivative.

  For $u \leq 9/5$, $\ys=1$, we find that $\mathrm{Rat}'(x)$ has
  the sign of
  \begin{equation}
    (3+u)(9-5u) + 2 (9-u)(3+5u) x + (27+234u-5u^2) x^2.
  \end{equation}
  It is straightforward to check that each coefficient of this
  expansion in $x$ is positive for $u \in [0,9/5)$, and only the first
  one vanishes for $u=9/5$. Hence, we have $\mathrm{Rat}'(x)>0$ for
  $x>0$, and the claim of the lemma is proved for $u \leq 9/5$.

  For $u > 9/5$, we use Remark~\ref{rem:uysinv} to recast
  $\mathrm{Rat}(x)$ as a rational function in $x$ with polynomial
  coefficients in $\ys \in (0,1)$. We now find that $\mathrm{Rat}'(x)$
  has the sign of
  \begin{equation}
    (1-\ys)(3-\ys)^3 + 2 (3-\ys)^2 (3+2\ys) x + (27 + 54 \ys + 18 \ys^2 + 2 \ys^3 - \ys^4) x^2.
  \end{equation}
  Again each coefficient of this expansion in $x$ is positive for
  $\ys \in (0,1)$; for the last one observe that the term $\ys^4$ is
  smaller than any of the other terms. So we also have
  $\mathrm{Rat}'(x)>0$ for $x>0$ and the claim of the lemma is now
  proved for all $u>0$.
\end{proof}

Taking the circle $|y|=\ys$ as integration contour in
\eqref{eq:genIn}, Lemma~\ref{lem:circlegood} enables us to obtain the
asymptotics of $I_n$ using Laplace's method. This method is covered in
classical textbooks such as~\cite{Dieudonne1986}, but we here give a
self-contained derivation of the estimate~\eqref{eq:genInasy}.
The general idea is to split the integration contour in two parts:
\begin{itemize}
\item a \emph{central arc} surrounding $\ys$ and of size tending to
  $0$ as $n \to \infty$, that captures the dominant contribution $I_n^{\mathrm{central}}$ to
  the integral,
\item the remaining \emph{tail arc}, whose contribution
  $I_n^{\mathrm{tail}}$ is exponentially negligible.
\end{itemize}
Here, we specifically take as central arc the intersection of the
integration circle with the cone
$\left\lvert\arg(y)\right\rvert<n^{\epsilon-1/2}$, for some small
$\epsilon>0$ which we will specify later. We work with the convention
that $\arg$ takes values between $-\pi$ and $\pi$.
\begin{figure}[h]
  \centering
  \begin{minipage}{.45\textwidth}
    \centering
    \includegraphics[width=\textwidth]{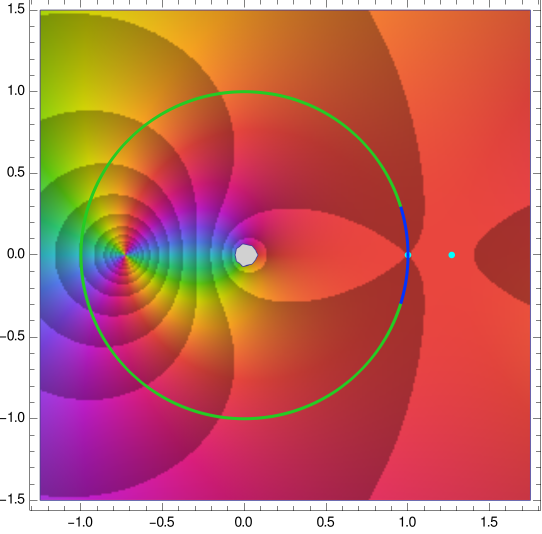}
  \end{minipage}
  \begin{minipage}{.45\textwidth}
    \centering
    \includegraphics[width=\textwidth]{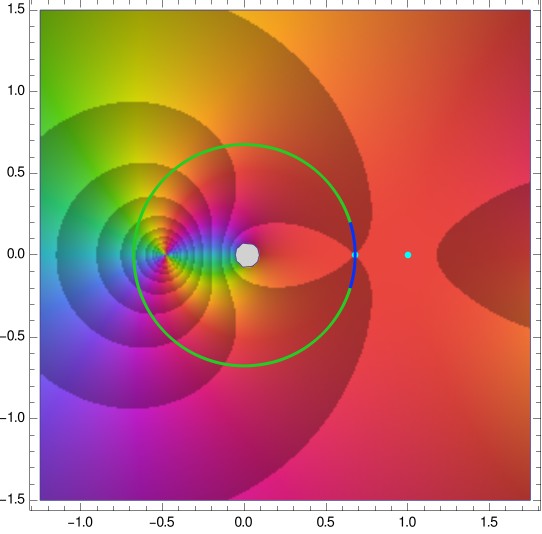}
  \end{minipage}
  \caption{The integration contour, in the $y$ complex plane,
    chosen for computing \eqref{eq:genIn},
    is the circle of radius $\ys$ centered at $0$.
    Left : $u=1.5$ hence $\ys=1$, a subcritical case.
    Right : $u=2.5$ hence $\ys \simeq 0.676$, a supercritical case.
    The integral $I_n^{\mathrm{central}}$ corresponding to the central arc
    $|\arg(y)| < n^{\epsilon - 1/2}$ (here $|\arg(y)| <0.3$, drawn in blue)
    asymptotically dominates over $I_n^{\mathrm{tail}}$ (corresponding to the arc drawn in green).
    The contour is here displayed on top of
    a colored representation of the complex function $\frac{1}{g(y)}$:
    the hue indicates the argument 
    and the shade varies with the modulus and marks its level curves.
    The function $\frac{1}{g(y)}$ has a pole at $0$,
    and two saddles seen on those plots (cyan dots), at $y=1$ and $y=3 \left( 1 - \sqrt{\frac{u-1}{u}} \right)$.
  }
  \label{fig:circleArcs}
\end{figure}

In the computations below, we set
\begin{equation}
  c := -\frac{\ys^2 g''(\ys)}{2g(\ys)}
\end{equation}
which is positive since $g(\ys)>0$ and $g''(\ys)<0$. Their explicit
expressions in the subcritical and supercritical phases have been
given above in \eqref{eq:gysaddle}, \eqref{eq:gsecond_ys_supercrit} and \eqref{eq:values_subcrit}. The quantity $c$ appears in the basic estimate
$\ln \frac{g(y)}{g(\ys)} \sim c (\arg (y))^2$ for $y \to \ys$.

\paragraph{Tail bound.} By the basic estimate above, there exists
$a \in (0,\pi]$ such that, for $\lvert\arg(y)\rvert < a$, we have
\begin{equation}
  \ln \left | \frac{g(y)}{g(\ys)} \right | \geq \frac {c} 2 \left |\arg (y)\right |^2
\end{equation}
Setting $c' = \frac{c a^2}{2 \pi^2}$ and using
Lemma~\ref{lem:circlegood}, we check straightforwardly that we have
\begin{equation}
  |g(y)| \geq g(\ys) e^{c' (\arg(y))^2}
  \label{eq:gycabound}
\end{equation}
for all $y$.  Using this lower bound for $\lvert\arg(y)\rvert \geq n^{\epsilon-1/2}$, we
immediately deduce that
\begin{equation}
  I_n^{\mathrm{tail}} = O\left( g(\ys)^{-n} e^{-c' n^{2\epsilon}} \right).
  \label{eq:tailboundsuper}
\end{equation}

\paragraph{Central approximation.} In the integral
$I_n^{\mathrm{central}}$, we perform the change of variable
$\nobreak{y=\ys e^{i s n^{-1/2}}}$, with $s$ ranging in the interval
$[-n^\epsilon,n^\epsilon]$. We have the Taylor expansion
\begin{equation}
  \log g(y) = \log g(\ys) + c \frac{s^2}n + O\left( \frac{s^3}{n^{3/2}} \right),
\label{eq:approxlog_g_noncrit}
\end{equation}
with $c>0$ as above, giving
\begin{equation}
  g(y)^{-n} = g(\ys)^{-n} e^{-c s^2} \left( 1 + O(n^{3\epsilon-1/2})\right)
  \label{eq:gyncentral}
\end{equation}
where the big $O$ is uniform in $s$. We shall therefore take
$\epsilon<1/6$ to have this error term tending to zero. Noting also
that
\begin{equation}
  f(y) \frac{dy}{ds}=\frac{i\ys f(\ys)}{n^{1/2}} \left( 1 + O(n^{\epsilon-1/2})\right)
\end{equation}
we get that
\begin{equation}
  I_n^{\mathrm{central}} = \frac{i\ys f(\ys)}{n^{1/2} g(\ys)^n} \int_{-n^\epsilon}^{n^\epsilon} ds \, e^{-c s^2} \left( 1 + O(n^{3\epsilon-1/2})\right).
\end{equation}
By dominated convergence the integral above converges for
$n \to \infty$ to
\begin{equation}
  \int_{-\infty}^{\infty} ds \, e^{-c s^2} = \sqrt{\frac{\pi}{c}}
\end{equation}
hence, by~\eqref{eq:tailboundsuper}, we have
\begin{equation}
  I_n \sim I_n^{\mathrm{central}} \sim \frac{i\ys f(\ys) \sqrt{\frac{\pi}{c}}}{n^{1/2} g(\ys)^n} =  i f(\ys) \sqrt{\frac{2\pi g(\ys)}{-g''(\ys)}}\, g(\ys)^{-n}\,  n^{-1/2}.
  \label{eq:integralresultsuper}
\end{equation}
This completes the proof of Proposition~\ref{prop:genericsaddle}.

\subsubsection{The critical case}
\label{sssec:critQ}

In the critical phase $u=9/5$, we need to use a different contour in the integral~\eqref{eq:Qnintysimp}.
Indeed, even though Lemma~\ref{lem:circlegood} holds, it is tedious to apply Laplace's method along the circle, since the integrand does not decay fast enough.

We shall rather use a contour that leaves the (now multiple) saddle point $\ys=1$
along directions of \emph{steepest descent} of $|g(y)|^{-1}$.
We choose for such a contour a shifted and scaled portion of the so-called \emph{Ceva trisectrix},
namely the curve parametrized by:
\begin{equation}
  y(\theta) := 1 - \frac{e^{i\theta} + e^{-i\theta} + e^{3i\theta}}{2}, \qquad \theta \in \left [- \frac \pi 3 ; \frac \pi 3 \right ] ;
  \label{eq:cevaTheta}
\end{equation}
see Figure~\ref{fig:CevaTrixArcs} for an illustration.
This somewhat arbitrary choice turns out to be convenient for explicit computations.
We have the following analog of Lemma~\ref{lem:circlegood}:
\begin{lem}
  As $y$ varies over the parametrized curve \eqref{eq:cevaTheta}, the function $y\mapsto |g(y)|$ admits a strict global minimum at $y(\theta=\pm \frac \pi 3)=\ys=1$, a strict global maximum at $y(\theta=0)=-\frac 1 2$, and varies monotonically over the two arcs connecting these extrema.
  \label{lem:cevagood}
\end{lem}

\begin{proof}
  We proceed very similarly to the proof of Lemma~\ref{lem:circlegood}:
  it is again convenient to use the rational parametrization $e^{i\theta} = \frac{(1 + i t)^2}{1+t^2}$ with $t$ now varying in the interval $[-\frac 1 {\sqrt{3}};\frac 1 {\sqrt{3}}]$.
  With this parametrization, the squared modulus of $y(\theta)$
  is again the square of a rational function $\mathrm{Rat}(t^2)$ in $t^2$,
  now with integer coefficients.
  We may straightforwardly verify on a computer algebra system that $\mathrm{Rat}$ is an increasing function on $[0;\frac 1 3]$ as wanted.
\end{proof}

We now apply Laplace's method to analyze the asymptotics of $Q_n$,
using the integral representation~\eqref{eq:Qnintysimp} with the above
contour. As in Section~\ref{sssec:noncritQ}, the general idea consists
in splitting the contour into a central arc surrounding the saddle
point $\ys=1$, and a remaining tail arc. More precisely,
viewing~\eqref{eq:cevaTheta} as a change of variable, we may write
\begin{equation}
  Q_n = \frac{J_n}{3i\pi n}, \qquad  J_n := \int_{-\pi/3}^{\pi/3} \,\frac{1-y(\theta)}{g(y(\theta))^n}\,  y'(\theta) d\theta.
  \label{eq:Jndef}
\end{equation}
Since $\ys$ is actually reached at the two endpoints $\pm \pi/3$ of
the interval for $\theta$, we split $J_n$ into three terms
$J_n^+ + J_n^- + J_n^{\mathrm{tail}}$, with:
\begin{itemize}
\item $J_n^+$ the contribution of the interval
  $\theta \geq \frac \pi 3 - n^{\epsilon-\frac 1 3}$ corresponding to
  the first half of the central arc,
\item $J_n^-$ the contribution of the interval
  $\theta \leq - \frac \pi 3 + n^{\epsilon-\frac 1 3}$ corresponding
  to the second half of the central arc,
\item and $J_n^{\mathrm{tail}}$ the remaining contribution of the tail arc.
\end{itemize}
\begin{figure}[h]
  \centering
  \begin{minipage}{.45\textwidth}
    \centering
    \includegraphics[width=\textwidth]{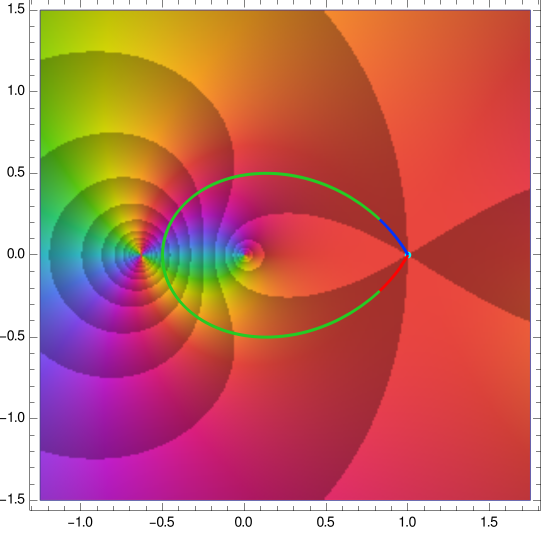}
  \end{minipage}
  \caption{The integration contour, in the $y$ complex plane,
    parametrized by Equation~\eqref{eq:cevaTheta}.
    We colored the Ceva trisectrix according to the subdivision in the text 
    (here for $n^{\epsilon - \frac{1}{3}} = .15$):
    the tail arc corresponding to $J_n^{\mathrm{tail}}$ in green,
    the first (respectively second) half of the central arc
    corresponding to $J_n^+$ (resp. $J_n^-$) in red (resp. blue).
    Again, the contour is displayed on top of a colored representation
    of the complex function $\frac{1}{g(y)}$,
    where the saddle $y=1$ is now a monkey saddle.
    Note that the central arcs reach that saddle point with an angle $\pm \frac{2 \pi}{3}$.}
  \label{fig:CevaTrixArcs}
\end{figure}

Again $\epsilon$ is a small positive real number and, as we will see
below, we may choose $\epsilon = 1/24$.  Note that we have
\begin{equation}
  J_n^+ = - \left( J_n^- \right)^*.
  \label{eq:Jnpmconj}
\end{equation}
We now mimic the reasoning
done at the end of Section~\ref{sssec:noncritQ}, using here the basic
estimate
\begin{equation}
  \ln \frac{g(y)}{g(\ys)} \sim \frac{(y-1)^3}4, \qquad y \to \ys = 1.
\end{equation}

\paragraph{Tail bound.} By the basic estimate above and by
Lemma~\ref{lem:cevagood}, we may find a positive constant $c$ such
that
\begin{equation}
  |g(y)| \geq g(\ys) e^{c |y-1|^3}
\end{equation}
for all $y$ on the integration contour.
Using this lower bound for $y$ on the tail arc, we deduce that there exists $c'>0$ such that
\begin{equation}
  J_n^{\mathrm{tail}} = O\left( g(\ys)^{-n} e^{-c' n^{3\epsilon}} \right).
  \label{eq:tailboundcrit}
\end{equation}

\paragraph{Central approximation.} Observe that, by
\eqref{eq:Jnpmconj}, it suffices to compute the large $n$ asymptotics
of $J_n^-$. Setting $\theta := -\frac{\pi}{3} + s n^{-1/3}$ with
$s \in [0, n^\epsilon]$, we have the Taylor expansion:
\begin{equation}
\log g(y) = \log \left( \frac{25}{432} \right) + \frac {3 \sqrt{3}} 4 \frac{s^3}{n} + O \left( \frac{s^4}{n^{4/3}} \right)
\label{eq:approxlog_g_crit}
\end{equation}
giving
\begin{equation}
  g(y)^{-n} = \left( \frac{25}{432} \right)^{-n} \exp \left( - \frac{3\sqrt{3}}{4} s^3 \right) \left( 1+ O\left( n^{4\epsilon - 1/3} \right) \right)
  \label{eq:gyncentralcrit}
\end{equation}
with the big $O$ uniform in $s$.  Taking $\epsilon = 1/24$ ensures that this error term tends to zero.
Noting also that
\begin{equation}
  (1-y)\frac{dy}{ds} = \frac{3 (1+ i\sqrt{3})}{2} \cdot \frac{s}{n^{2/3}} \left( 1+O(n^{\epsilon-1/3}) \right),
\end{equation}
we obtain that
\begin{equation}
  J_n^- = \frac{3 (1+ i\sqrt{3})}{2 n^{2/3}} \left( \frac{25}{432} \right)^{-n} \int_{0}^{n^\epsilon} ds \exp \left( - \frac{3\sqrt{3}}{4} s^3 \right) s \left( 1+ O\left( n^{4\epsilon - 1/3} \right) \right).
\end{equation}
By dominated convergence the integral above converges for
$n \to \infty$ to
\begin{equation}
  \int_{0}^{\infty} ds \exp \left( - \frac{3\sqrt{3}}{4} s^3 \right) s = \frac{2^{4/3} \Gamma(2/3)}{9}
\end{equation}
so that
\begin{equation}
  J_n^- \sim \frac{(1+ i\sqrt{3})\, 2^{1/3} \Gamma(2/3)}{3\, n^{2/3}} \left( \frac{25}{432} \right)^{-n}.
\end{equation} 
Using~\eqref{eq:Jndef}, \eqref{eq:Jnpmconj}, and~\eqref{eq:tailboundcrit}, we deduce
\begin{equation}
  Q_n = \frac 1 {3 i \pi n} J_n \sim \frac{2^{4/3} \Gamma(2/3)}{3\sqrt{3}\, \pi\, n^{5/3}} \left( \frac{25}{432} \right)^{-n}
  \label{eq:integralresultcrit}
\end{equation}
yielding the wanted asymptotic behavior~\eqref{eq:Qn_asymcrit} by noting that $\Gamma(2/3)\Gamma(1/3) = 2 \pi/\sqrt{3}$.

\section{Asymptotic analysis of the two-point function}
\label{sec:asymptKl}

In this section, we adapt the saddle-point method of
Section~\ref{sec:asymptQ} to study the asymptotics of the
distance-dependent two-point function. More precisely, fixing again $u$
to a nonnegative real value, let us denote by $K_\ell^{(n)} = K_\ell^{(n)}(u)$ the
coefficient of $g^n$ in the series $K_\ell(g,u)$ introduced in
Section~\ref{ssec:blockweightedTPF}. In other words, $K_\ell^{(n)}$ is
the sum of the weight $u^{b(\cQ)-1}$ over all rooted quadrangulations
$\cQ$ with $n$ faces and an additional marked edge\footnote{Recall that we also require that the marked edge and the root edge do not have the same two endpoints.} at root distance
$\ell$.
Introducing the quantity $K^{(n)} = K^{(n)}(u) := \sum_{\ell \geq 1} K_{\ell}^{(n)} = [g^n]\hat K(1,g,u)$, it is natural to consider the ratio 
\begin{equation}
  p_{\ell}^{(n)} = \frac{K_{\ell}^{(n)}}{K^{(n)}}
\end{equation}
which is the probability that, in a random block-weighted quadrangulation
with $n$ faces and an additional marked edge,
the root distance of the marked edge is $\ell$.
We will see that, for any $u \geq 0$,
there exists an exponent $\nu$ (depending on $u$) and a non-trivial function $\rho_u$ such that
\begin{equation}
  n^{\nu} p_{\lfloor d \, n^{\nu} \rfloor}^{(n)} \to \rho_u(d)
\end{equation}
uniformly for $d$ on any compact set of $\mathbb R_{+}^{*}$.
The limit $\rho_u$ is a continuous probability density on $\mathbb R_{+}$,
which we call the \emph{distance profile}.
More precisely, our main asymptotic result is summarized in the following proposition:
\begin{prop}
  The distance profile $\rho_u$ is as follows:
  \begin{itemize}
    \item in the subcritical phase $0 \leq u \leq 9/5$, we have $\nu = 1/4$ and
      \begin{equation}
        \rho_u(d) = \frac{1}{D} \  \rho_{\mathrm{sub}}\!\left( \frac{d}{D} \right),
  \qquad D = \left( \frac{9-5u}{3+u} \right)^{1/4}
        \label{eq:distanceprofile_subcrit}
      \end{equation}
      with
      \begin{equation}
        \rho_{\mathrm{sub}}(x) = \frac{8}{x \sqrt{\pi}}
        \int_{0}^{+\infty} s^2 (2 s^2 - 3) e^{- s^2}
        \left( 1 - 6 \frac{1 - \cosh\left( x\sqrt{3 s} \right) \cos\left( x\sqrt{3 s} \right)}{\left( \cosh\left( x\sqrt{3 s} \right) - \cos\left( x\sqrt{3 s} \right) \right)^2} \right)ds;
      \end{equation}
    \item in the critical phase $u = 9/5$, we have $\nu = 1/3$ and
      \begin{equation}
        \rho_u(d)  = \rho_{\mathrm{crit}}(d) = 3^{1/3} \Gamma(1/3) \frac{d}{D^2} \Airy \left(\frac{d}{D} \right),
        \qquad D = \frac{2 \pi^2 - 15}{6^{2/3}}
        \label{eq:distanceprofile_crit}
      \end{equation}
      with $\Airy$ the Airy function;
    \item in the supercritical phase $u \geq 9/5$, we have $\nu = 1/2$ and
      \begin{equation}
        \rho_u(d) = \frac{1}{D} \rho_{\mathrm{sup}}\!\left( \frac{d}{D} \right),
        \qquad D = \frac{\partial \hat h}{\partial
        t}(1,\zs) \frac{\sqrt{(3-\ys)^5(1-\ys)}}{6 \sqrt{6} \, \ys}
        \label{eq:distanceprofile_supercrit}
      \end{equation}
      with
      \begin{equation}
        \rho_{\mathrm{sup}}(x) = x e^{-\frac{1}{2} x^2}
        \label{eq:rho_sup_Rayleigh_distance_profile}
      \end{equation}
      the Rayleigh distribution.
  \end{itemize}
  \label{prop:distanceprofiles}
\end{prop}
\begin{rem}
  The scaling factor $D$ in \eqref{eq:distanceprofile_supercrit}
  can be made slightly more explicit by noting that,
  from Equation~\eqref{eq:hhatmeroalt} with $\xs=x(\ys)=\frac{3-\ys-\sqrt{9-6\ys-3\ys^2}}{2\ys}$, we have :
  \begin{equation}
    \begin{split}\frac{\partial \hat h}{\partial t}(1,\zs) 
      &= \frac{27(1-\ys)(3+\ys)}{\ys^2(3-\ys)^2} \sum_{k \geq 1} k \frac{\xs^{3 k}}{1-\xs^k} \\
    &= \frac{27(1-\ys)(3+\ys)}{\ys^2(3-\ys)^2} \sum_{m \geq 3} \frac{1}{2 T_m\left( \frac{3}{2 \ys} - \frac{1}{2} \right) - 2}\end{split}
    \label{eq:dhathtz}
  \end{equation}
  with $T_m$ the $m$-th Chebychev polynomial of the first kind.
  Note that at large $u$, $\ys \to 0$ and the term $m=3$ dominates in \eqref{eq:dhathtz},
  and thus $\frac{\partial \hat h}{\partial t}(1,\zs) \sim \frac{27 \times 3}{3^2 \ys^2} \frac{\ys^3}{27}$ and $D \underset{u\to \infty}{\longrightarrow} \frac{1}{2\sqrt{2}}$.
  \label{rem:scaling_supercrit_limit}
\end{rem}
Figure~\ref{fig:scaling_factors_and_distance_profiles_Kl} presents a plot of the scaling factors $D$ in the different regimes, and displays instances of the distance profile in the three regimes.
\begin{figure}[h]
  \centering
  \includegraphics[width=.49\textwidth]{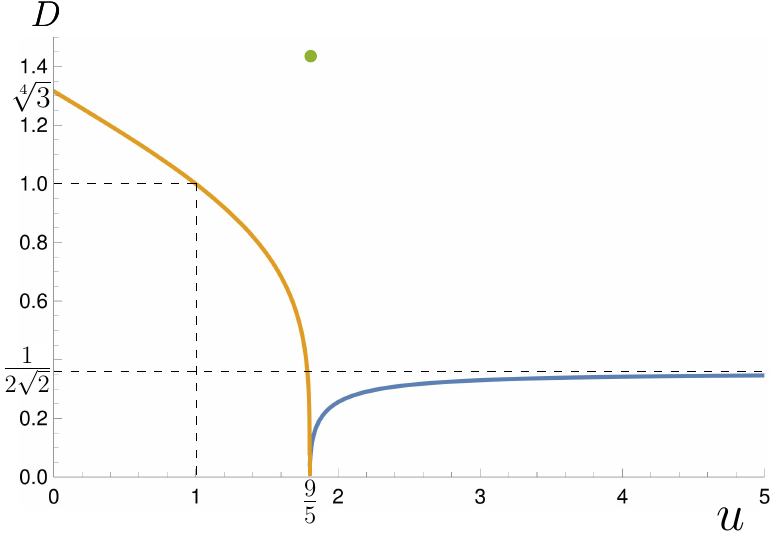}
  \includegraphics[width=.49\textwidth]{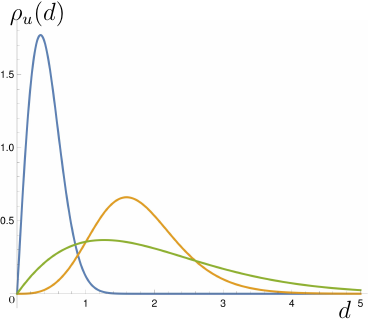}
  \caption{Left: the scaling factors $D$ for the different regimes of Proposition~\ref{prop:distanceprofiles}. \\
  Right: the distance profile $\rho_{u}(d)$ for $u = 1$ (subcritical, in orange), $u = 9/5$ (critical, in green), and $u = 4$ (supercritical, in blue).}
  \label{fig:scaling_factors_and_distance_profiles_Kl}
\end{figure}
The proof of Proposition~\ref{prop:distanceprofiles} is the topic of the coming subsections.
Our starting point is the contour integral representation
\begin{equation}
  K_\ell^{(n)} = \frac{1}{2i\pi} \oint \frac{K_\ell(g,u)dg}{g^{n+1}}
  \label{eq:kellnu}
\end{equation}
where, as in~\eqref{eq:Qnintbasic}, the contour of integration is
initially a small circle around the origin. Indeed, since
$K_\ell(g,u)$ is the generating function of rooted quadrangulations
with an additional marked edge at root distance $\ell$, its radius of
convergence in $g$ cannot be smaller than that of $Q(g,u)$, which is
nonzero for any $u$ as discussed above. So, again, the contour
integral makes sense.

Before entering into details, let us sketch our general
strategy. First, we use the rational parametrization $g=g(y)$ of
Proposition~\ref{prop:Qratparam} to rewrite~\eqref{eq:Qnintbasic} as
an integral over $y$. We then deform the contour of integration in $y$
as in Section~\ref{sec:asymptQ}, so that it passes through the
saddle-point $\ys$ of $g(y)$. The dominant contribution to the
integral comes from the vicinity of the saddle-point: to compute it we
shall find an estimate of $K_\ell$ in this region, with $\ell$ itself
scaling as a power of $n$. This is what we shall call \emph{key
  estimates} in the following.  The difficulty is that we only know
$K_\ell$ through its generating function $\hat K(t)$: to overcome
this, we write $K_\ell$ itself as a contour integral over $t$, whose
asymptotics for $y$ near $\ys$ is related with the singular behavior
of $\hat{K}(t)$ near $t=1$. It turns out that this behavior changes
drastically at the phase transition, which forces us to study the
three subcritical, critical and supercritical cases separately. In the
two latter cases, $\hat{K}(t)$ has a single dominant pole near $t=1$,
while in the first case infinitely many poles contribute.

\subsection{Analytic preliminaries}
\label{ssec:asymptKl_prelim}

To perform the asymptotic analysis of the contour
integral~\eqref{eq:kellnu}, we shall combine many of the ingredients
introduced above. First, we use the
rational parametrization $g=g(y)$ of Proposition~\ref{prop:Qratparam},
to yield
\begin{equation}
  K_\ell^{(n)} = \frac{1}{2i\pi} \oint \frac{K_\ell(g(y),u)g'(y)dy}{g(y)^{n+1}}
  \label{eq:Klnuy_integraldef}
\end{equation}
with initial contour a small circle around the origin. Second, we have by~\eqref{eq:defhatK}
\begin{equation}
  K_\ell(g(y),u) = \frac{1}{2i\pi} \oint \frac{\hat K(t,g(y),u)dt}{t^\ell}
  \label{eq:Kellgyu}
\end{equation}
where, by Theorem~\ref{thm:hatKexpr} and Remark~\ref{rem:zy_param}, $\hat K(t,g(y),u)$ reads
\begin{equation}
  \hat K(t,g(y),u) = \frac{z(y)}{g(y)} \cdot \frac{\hat h(t, z(y)) - 1}{u + (1-u) \hat h (t, z(y))}, \qquad \frac{z(y)}{g(y)} =  \left(1 + u \, \frac{y(2-y)}{3}\right)^2.
  \label{eq:Ktgyu}
\end{equation}
Note that, in~\eqref{eq:Kellgyu}, we may take the contour of
integration to be the circle $|t|=1$ for $y$ small enough, since the
series $\hat K(t,g(y),u)$ counting quadrangulations with an additional
marked edge is then absolutely convergent. As for the quantity
$\hat h(t, z(y))$ appearing in the above display, it is explicitly
given by the following proposition:

\begin{prop}
  \label{prop:hathzy}
  We have
  \begin{equation}
    \hat h(t,z(y)) = \frac{27(1-y)(3+y)}{y^2(3-y)^2} \sum_{k \geq 1} k  x(y)^{2 k} \frac{1-x(y)^k}{1-t x(y)^k}
    \label{eq:hathtzy}
  \end{equation}
  with
  \begin{equation}
    x(y) = \frac{3-y-\sqrt{9-6y-3y^2}}{2y} = \sum_{n \geq 0} m_n \left( \frac{y}{3} \right)^{n+1},
    \label{eq:xyexplicit}
  \end{equation}
  $m_n$ denoting the $n$-th Motzkin number.
\end{prop}

\begin{proof}
  We start from the expression~\eqref{eq:hathexpr} for $\hat
  h(t,z)$. We claim that, substituting $z=z(y)$, the series $x$ and
  $r$ appearing in this expression evaluate respectively to $x(y)$ as
  above, and to
  \begin{equation}
    r(y)=\frac 3 {3-y}.
    \label{eq:ryexp}
  \end{equation}
  Indeed, introducing the series $w:=z r^2$, \eqref{eq:rxzeqs} allows
  us to write $r = 1 + w r$, hence $r = \frac{1}{1-w}$ and thus
  $z = \frac{w}{r^2} = w (1-w)^2$.  Substituting $z=z(y)$, $w$
  evaluates to $w(y)$ determined by the conditions
  $z(y) = w(y) (1-w(y))^2$ and $w(y)=O(y)$. Comparing with the
  expression $z(y)=\frac{y}{3}\left( 1-\frac{y}{3} \right)^2$ of
  Remark~\ref{rem:zy_param}, it jumps to the eyes that
  $w(y)=\frac{y}{3}$.  We immediately deduce the above expression for
  $r(y)$.  Now, from $x + \frac{1}{x} + 1 = \frac{1}{w}$, we find that
  $x(y)$ satisfies
  \begin{equation}
    x(y) = \frac{y}{3}(1 + x(y) + x(y)^2).
  \end{equation}
  We recognize the classical equation satisfied by the generating
  function of Motzkin numbers, in the variable $y/3$, which
  establishes~\eqref{eq:xyexplicit}. It is straightforward to check
  that the prefactor $r^2 \frac{(1-x^2)^2}{x^2}$
  in~\eqref{eq:hathexpr} evaluates for $r=r(y)$ and $x=x(y)$ to that
  in~\eqref{eq:hathtzy}.
\end{proof}

A priori $\hat h(t,z(y))$ is defined for $y$ small, but the above
proposition provides an interesting analytic continuation. Indeed,
with the standard determination of the square-root, $x(y)$ extends to
an analytic function over the domain
\begin{equation}
  Y := \C \setminus  \left( (-\infty,-3] \cup [1,+\infty) \right).
\end{equation}
We claim that $|x(y)| < 1$ for all $y \in Y$: indeed, from the relation $x(y)+x(y)^{-1}=\frac3y-1$ it appears that $|x(y)| \neq 1$ for all $y \in Y$, and we have
$|x(y)|<1$ for $y$ small. From~\eqref{eq:hathtzy}, we deduce that $\hat h(t,z(y))$
extends to a bivariate meromorphic function over $\C \times Y$. For a
fixed $y \in Y$, its poles are at
$t=x(y)^{-1}, x(y)^{-2}, x(y)^{-3}, \ldots$.

Now, by~\eqref{eq:Ktgyu}, the function
$(t,y) \mapsto \hat K(t,g(y),u)$ is also meromorphic over
$\C \times Y$. Its singularities in this domain correspond to the solutions to the equation
\begin{equation}
  \varphi_y(t) := \hat h (t, z(y)) - \frac{u}{u-1} = 0.
  \label{eq:ktgyupoles}
\end{equation}
For any fixed $y$ such that $t=0$ is not a root of this equation,
$t \mapsto \hat K(t,g(y),u)$ admits a power series expansion, whose
coefficients are given by the contour integral~\eqref{eq:Kellgyu}
taken over a small circle around the origin. From this integral
representation, we see that the coefficients are holomorphic in
$y$. Noting that
\begin{equation}
  \hat h (0, z(y)) = h_1(z(y)) = \frac{9-3y-y^2}{(3-y)^2},
\end{equation}
is real if and only if $y$ is real, varies between $\frac14$ and
$\frac54$ for $y$ going from $-3$ to $1$, and takes the value $1$ at
$y=0$, we see that the equation $\hat h (0, z(y)) = \frac{u}{u-1}$
admits no solution $y \in Y$ for $u \in [0,5]$, and a unique solution
$y^*\in (0,1)$ for $u>5$ (i.e.\ $\frac{u}{u-1}<\frac54$). We deduce
that, for any $\ell \geq 1$, the function $y \mapsto K_\ell(g(y),u)$
is analytic in $Y$ for $u \leq 5$, and in $Y \setminus \{y^*\}$ for
$u>5$. In this latter case, we may see that the pathological value
$y^*$ corresponds to a pole of $K_\ell(g(y),u)$, but for our purposes
we only need to note that $y^*$ is larger than the saddle-point $\ys$
of \eqref{eq:ysaddle}. This follows from an elementary reasoning using
the relation $\frac{u}{u-1}=\frac{9}{(3-\ys)^2}$ valid in the
supercritical regime.

\subsection{The supercritical case}
\label{ssec:supercritKl}

In this section, we assume $u>9/5$. From the previous discussion, we
see that the integration contour for $y$
in~\eqref{eq:Klnuy_integraldef} can be deformed into the circle
$|y|=\ys \in (0,1)$, as was done in Section~\ref{sssec:noncritQ} to
analyze $Q_n$. We start with a lemma about the analytic properties of
$t \mapsto \hat K(t, g(y), u)$ for $y$ close to $\ys$.
\begin{lem}
  \label{lem:t0yexpansuper}
  Set $\xs := x(\ys) \in (0,1)$, $\zs := z(\ys)$, and
  $\mathcal D:=\{t:|t| \leq \xs^{-1/2}\}$. There exists a neighborhood $N(\ys)$
  of $\ys$ whose each element $y$ is such that the meromorphic
  function $t \mapsto \hat K(t, g(y), u)$ admits in $\mathcal D$ a unique pole
  $t_0(y)$, which is simple and does not belong to $\partial \mathcal D$. For
  $y \to \ys$ we have
  \begin{equation}
    t_0(y) = 1 - C\left( \frac{y}{\ys}-1 \right) + O\left( (y-\ys)^2 \right),
    \qquad C := \frac{18 \, \ys}{(3-\ys)^3 \frac{\partial \hat h}{\partial t}(1,\zs)}.
    \label{eq:t0yexpan}
\end{equation}
\end{lem}

\begin{figure}[h]
  \centering
  \includegraphics{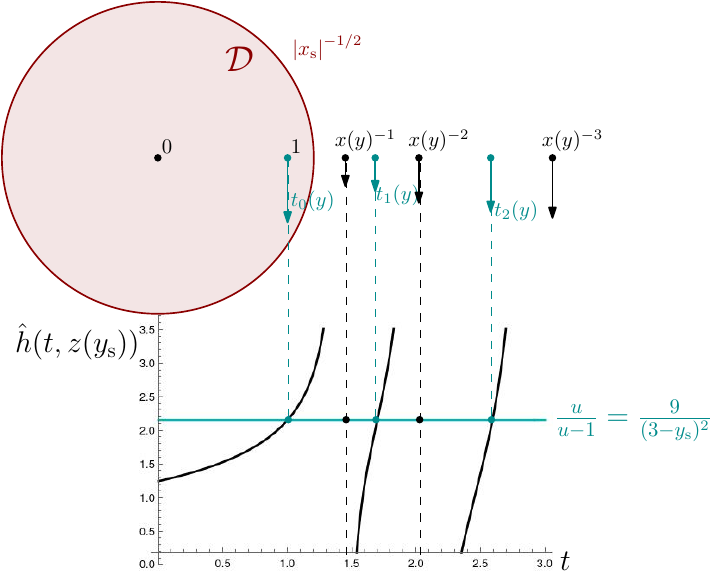}
  \caption{Top: representation in the complex $t$ plane of the domain $\mathcal D$,
    of the poles $x(y)^{-i}$ of $\hat h(t, z(y))$, $i = 1,2,3,\ldots$,
    and of the poles $t_{m}(y)$ of $\hat K(t,g(y),u)$ for $m = 0,1,2,\ldots$ (cyan).
    The dots correspond to $y=\ys$, and the arrows the first order displacement
    for $y = \ys + i \varepsilon$ with small positive $\varepsilon$,
    as in the central arc of the integration contour in $y$.
    Bottom: the (real) values of the function $\hat h(t,z(\ys))$ for real $t$.
    The $t_m(\ys)$ correspond to the crossing of this function
    with the horizontal line with ordinate $\frac{u}{u-1}$.
    Here $u\simeq 1.86$.
    For $y$ in a neighborhood $N(\ys)$,
    there is a unique pole $t_0(y)$ in the domain $\mathcal D$.
  }
  \label{fig:localisation_poles_super}
\end{figure}

\begin{proof}
  See Figure~\ref{fig:localisation_poles_super} for a graphical reference.
  As seen above, the poles of $t \mapsto \hat K(t, g(y), u)$ correspond to the solutions
  of~\eqref{eq:ktgyupoles}, i.e.\ to the zeros of $\varphi_y$.

  Let us first consider the case $y=\ys$. By
  Remark~\ref{rem:hath_t_eq_1},~\eqref{eq:ryexp},
  and~\eqref{eq:ysaddle}, we have
  \begin{equation}
    \hat h (1,z(\ys)) = r(\ys)^2 = \left( \frac{3}{3-\ys} \right)^{2}= \frac{u}{u-1}.
    \label{eq:hath_1_ys}
  \end{equation}
  Thus, $1$ is a zero of $\varphi_{\ys}$. Furthermore,
  by~\eqref{eq:hathtzy}, $\hat h(t, z(\ys))$ is of the form
  $\sum_{k \geq 1} \frac{\beta_k}{1-t \xs^{k}}$ with the $\beta_k$
  positive, which has the following consequences:
  \begin{itemize}
  \item On the real interval $(-\infty,\xs^{-1})$, $\varphi_{\ys}$ is
    increasing and has a nonzero derivative: $1$ is its only (simple)
    zero in this interval.
  \item If $t$ has a non-zero imaginary part, then $\varphi_{\ys}(t)$
    has a non-zero imaginary part of the same sign: there are no zeros of $\varphi_{\ys}$
    outside the real axis.
  \end{itemize}
  This reasoning shows that $t=1$ is the unique, simple, pole of
  $\hat K (t, g(\ys), u)$ inside $\mathcal D$.

  Now, for $y$ close to $\ys$, we argue by deformation.  First, since
  $\frac{\partial \hat h}{\partial t}(t,\zs) = \sum_{k \geq 1}
  \frac{\beta_k \xs^k}{(1-t \xs^{k})^2}$ is positive at $t=1$, the
  implicit function theorem ensures that $\varphi_y$ admits for $y$
  close to $\ys$ a zero $t_0(y)$ close to $1$, varying
  analytically. Let us show that this is the unique zero inside
  $\mathcal D$. Denote by $\tilde{\varphi}_y$ the restriction of $\varphi_y$ to
  $\partial \mathcal D$. For $y$ varying in a small closed ball $B$ around
  $\ys$, the Heine-Cantor theorem entails that the family
  $(\tilde{\varphi}_y)_{y \in B}$ is continuous with respect to the
  uniform norm. Since $|\tilde{\varphi}_{\ys}|$ remains bounded away
  from $0$, we have
  $|\tilde{\varphi}_y-\tilde{\varphi}_{\ys}|<|\tilde{\varphi}_{\ys}|$
  for $y$ in some neighborhood $N(\ys)$ of $\ys$, and Rouché's theorem allows us to conclude that
  $\varphi_y$ and $\varphi_{\ys}$ have then the same number of zeros,
  as wanted.

  The series expansion~\eqref{eq:t0yexpan} is obtained by
  differentiating the relation $\varphi_y(t_0(y))=0$ to yield
  \begin{equation}
    C = - \ys t'_0(\ys) = \frac{\ys}{\frac{\partial \hat h}{\partial t} (1,\zs)} \left. \frac{\partial \hat h(1,z(y))}{\partial y} \right\vert_{y=\ys}.
  \end{equation}
  We get the slightly more explicit expression for $C$ given in the
  statement upon writing
  $\frac{\partial \hat h(1,z(y))}{\partial y}
  \overset{\eqref{eq:hath_t_eq_1}}{=} \frac{d r(y)^2}{d y}
  \overset{\eqref{eq:ryexp}}{=} \frac{18}{(3-\ys)^3}$.
  \end{proof}

From this lemma we get a key estimate for $K_\ell(g(y),u)$ near
$\ys$. Namely, for any $y \in N(\ys)$, deforming the
contour in~\eqref{eq:Kellgyu} and applying the residue theorem, we get
\begin{equation}
  \begin{split}
    K_\ell(g(y),u) &= - \mathrm{Res}_{t=t_0(y)} \frac{\hat K(t,g(y),u)dt}{t^\ell} + \frac{1}{2i\pi} \oint_{\partial \mathcal D} \frac{\hat K(t,g(y),u)dt}{t^\ell} \\
    &= \frac{z(y)}{g(y)} \cdot \frac{\hat h(t_0(y), z(y)) - 1}{(u-1) \frac{\partial \hat h}{\partial t}(t_0(y),z(y))} \cdot \frac{1}{t_0(y)^{\ell}} + O\left( \xs^{\ell/2} \right) \\
    &= \frac{z(y)}{g(y)} \cdot \frac{1}{(u-1)^2 \  \frac{\partial \hat h}{\partial t}(t_0(y),z(y))} \cdot \frac{1}{t_0(y)^{\ell}} + O\left( \xs^{\ell/2} \right)
  \end{split}
  \label{eq:Kellresiduesuper}
\end{equation}
where we get to the last line using $\hat h(t_0(y), z(y)) - 1=\frac{1}{u-1}$. Note that the big $O$ is uniform in any compact subset of $N(\ys)$. It is interesting to see that this estimate allows to read off informally the relevant scaling regime for $\ell$. Indeed, from~\eqref{eq:t0yexpan}, we see that $\ell$ should scale as the inverse of $|y-\ys|$, which itself scales as $\sqrt{n}$ from the discussion of the central approximation in Section~\ref{sssec:noncritQ}. We make this intuition rigorous and precise in the following result.

\begin{prop}
  \label{prop:K_ellnasympt}
  For any $d>0$, we have as $n \to \infty$
  \begin{equation}
    K_{\lfloor d n^{1/2} \rfloor}^{(n)} = \frac{4}{9} \ys \sqrt{\frac3\pi(1-\ys)(3-\ys)} \left( \frac{432}{\ys(6-\ys)^2} \right)^n n^{-1} \times \frac{d}{D^2} e^{-\frac{1}{2} \left( \frac{d}{D} \right)^2} \times \left(1+o(1)\right)
    \label{eq:K_ellnasympt}
  \end{equation}
  with
  $D^2:=\frac{2 c}{C^2} = \left( \frac{\partial \hat h}{\partial
      t}(1,\zs) \right)^2 \frac{(3-\ys)^5(1-\ys)}{216 \, \ys^2}$.
\end{prop}

\begin{proof}
  Let us set $\ell := \lfloor d n^{1/2} \rfloor$ throughout this
  proof. Recall that $K_{\ell}^{(n)}$ is given by the contour integral
  representation~\eqref{eq:Klnuy_integraldef}, taken along the circle
  $|y|=\ys$. To analyze this integral, we proceed similarly to
  Section~\ref{sssec:noncritQ}, but since the integrand has a more
  complicated form, we now split the circle into \emph{three} parts:
  \begin{itemize}
  \item the \emph{central arc} given again by
    $\left\lvert\arg(y)\right\rvert<n^{\epsilon-1/2}$, for some
    $\epsilon \in (0,\frac16)$,
  \item a \emph{near tail arc} remaining with the neighborhood
    $N(\ys)$ of Lemma~\ref{lem:t0yexpansuper}: we take it of the form
    $n^{\epsilon-1/2} < \left\lvert\arg(y)\right\rvert < a$ for some
    fixed $a$ independent of $n$ and $\ell$,
  \item a remaining \emph{far tail arc} given by $\left\lvert\arg(y)\right\rvert > a$.
  \end{itemize}
  
  \paragraph{Tail bounds.} Let us show that the contributions of the
  tail arcs to~\eqref{eq:Klnuy_integraldef} are exponentially smaller
  than $g(\ys)^{-n}$.
  
  We start with the far tail arc: by a compactness argument we may find
  $\tau>0$ (possibly smaller than $1$) such that
  $t \mapsto \hat K(t,g(y),u)$ is holomorphic in the disk
  $|t|\leq \tau$ for all $y$ on the arc. By~\eqref{eq:Kellgyu} we
  deduce that $K_\ell(g(y),u)=O(\tau^{-\ell})$ uniformly in $y$. Using
  Lemma~\ref{lem:circlegood}, we find that the contribution of
  the far tail arc is $O\left(\tau^{-\ell} g(\ys e^{ia})^{-n}\right)$
  which, since $\ell$ is taken of order $\sqrt{n}$, is indeed
  exponentially smaller than $g(\ys)^{-n}$.

  As for the near tail arc, we use the
  estimate~\eqref{eq:Kellresiduesuper}. The error term
  $O(\xs^{\ell/2})$ is of no concern. In the series
  expansion~\eqref{eq:t0yexpan}, the factor $\frac{y}{\ys}-1$ is
  purely imaginary at first order as we follow the circle $|y|=\ys$,
  and therefore we may find a constant $c''$ such that
  $\ln |t_0(y)| \geq -c'' \arg(y)^2$ for all $y$ on the near tail arc.
  Combining with~\eqref{eq:gycabound} and~\eqref{eq:Kellresiduesuper},
  we find that the modulus of the integrand
  in~\eqref{eq:Klnuy_integraldef} is bounded by a constant times
  $e^{(-c' n+c'' \ell) \arg(y)^2} g(\ys)^{-n}$. As
  $\arg(y)^2 \geq n^{2\epsilon-1}$, the contribution of the near tail
  arc is indeed exponentially smaller than $g(\ys)^{-n}$.

  \paragraph{Central approximation.} As in Section~\ref{sssec:noncritQ},
  we perform the change of variable $y=\ys e^{i \, s \, n^{-1/2}}$
  with $s$ ranging in the interval
  $[-n^{\epsilon},n^{\epsilon}]$. From the
  expansion~\eqref{eq:t0yexpan}, we find that
  \begin{equation}
    t_0(y)^{-\ell} = e^{i\, C \, d \, s} + O(n^{2\epsilon-1/2})
  \end{equation}
  where the error is uniform in $y$. Plugging into the
  estimate~\eqref{eq:Kellresiduesuper}, we get
  \begin{equation}
    \begin{split}
      K_\ell(g(y),u) &= \frac{z(\ys)}{g(\ys)} \cdot \frac{1}{(u-1)^2 \  \frac{\partial \hat h}{\partial t}(1,z(\ys))} \cdot e^{i\, C \, d \, s} + O(n^{2\epsilon-1/2}) \\
      &= \frac{8 \ys (3-\ys)}9 \cdot C e^{i\, C \, d \, s} + O(n^{2\epsilon-1/2})
    \end{split}
  \end{equation}
  where the prefactor is simplified using the above expressions for
  $z(\ys),g(\ys),u,C$ in terms of $\ys$.  Combining
  with~\eqref{eq:gyncentral}, noting furthermore that
  \begin{equation}
    \frac{g'(y)}{g(y)} \cdot \frac{dy}{ds} = \frac{d \log g(y)}{ds} = \frac{2 c s}{n} + O(n^{2\epsilon-1}),
    \label{eq:gy_derivlog_trick}
  \end{equation}
  with $c$ as in Section~\ref{sssec:noncritQ}, and invoking again the
  dominated convergence theorem, we find that the contribution of the
  central arc to~\eqref{eq:Klnuy_integraldef} is asymptotically
  equivalent to
  \begin{equation}
    \frac{8 \ys (3-\ys)}{9i\pi\, n\, g(\ys)^n}\int_{-\infty}^\infty cs \, C e^{i\, C \, d \, s - c s^2} ds = \frac{8 \ys (3-\ys)}{9\, n\, g(\ys)^n} \sqrt{\frac{c}{\pi}} \times \frac{d}{D^2} e^{-\frac{1}{2} \left( \frac{d}{D} \right)^2}
  \end{equation}
  with $D^2=\frac{2c}{C^2}$ as in the statement of the theorem. Noting
  finally that $c=\frac34 \cdot \frac{1-\ys}{3-\ys}$ in the
  supercritical phase, we get precisely the right-hand side
  of~\eqref{eq:K_ellnasympt}. As seen above the tail arcs give a
  negligible contribution, hence the result is established.
\end{proof}

\begin{rem}
  It appears from the above proof that the $o(1)$ appearing
  in~\eqref{eq:K_ellnasympt} is uniform for $d$ varying in any compact
  subset of $(0,\infty)$.
\end{rem}

\begin{cor}  
  The distance profile $\rho_u(d)$ is given by
  \begin{equation}
    \rho_u(d) = \lim_{n\to \infty} n^{1/2} \frac{K_{\lfloor d \, n^{1/2} \rfloor}^{(n)}(u)}{K^{(n)}(u)} = \frac{d}{D^2} e^{-\frac{1}{2} \left( \frac{d}{D} \right)^2}
    \label{eq:rho_super}
  \end{equation}
  which is a well normalized probability distribution on $[0,\infty)$.
\end{cor}

\begin{proof}
From~\eqref{eq:hatK1_combi}, we get the estimate 
\begin{equation}
  K^{(n)} = (2n-1) \, Q_n - u [g^n] Q(g,u)^2 \sim 2n \, Q_n
\end{equation}
Indeed, $Q(g,u)^2$ has the same square-root singularity as $Q(g,u)$ for $g \to g(\ys)$,
hence $\frac{[g^n] Q(g,u)^2}{Q_n} = O(1)$.
Note that we could also obtain the asymptotics for $K^{(n)}$ from Remark~\ref{rem:hath_t_eq_1}
with a contour integral, as in Section~\ref{sec:asymptQ}.
Then, from Equation~\eqref{eq:Qn_asymsuper} of Proposition~\ref{prop:Qn_asym},
we find 
\begin{equation}
  K^{(n)} \sim \frac{4}{9} \ys \sqrt{\frac3\pi(1-\ys)(3-\ys)} \left( \frac{432}{\ys(6-\ys)^2} \right)^n n^{-1/2},
\end{equation}
and the expression for the distance profile $\rho_u(d)$ in \eqref{eq:rho_super} and \eqref{eq:distanceprofile_supercrit} follows.
It is easily checked that $\rho_u(d)$ integrates to $1$.
\end{proof}

\subsection{Approximations of the auxiliary function $\hat{h}$ near its singularity}
\label{ssec:hsingtext}

We now consider the case $u \leq 9/5$. This poses a new challenge,
since the saddle point of the function $g(y)$, ruling the large $n$
behavior of the integral~\eqref{eq:Klnuy_integraldef}, is at $\ys=1$,
which corresponds to a singularity of the bivariate function
$\hat h(t,z(y))$. Namely, from Proposition~\ref{prop:hathzy} we see
that $x(y)$ is singular and tends to the value $1$, which means that
the poles $(x(y)^{-k})_{k \geq 1}$ of $t \mapsto \hat h(t,z(y))$
densify along a logarithmic spiral. The purpose of this section is to
provide precise approximations of the function $\hat h$ in this
limit. We only state them here, the proofs being given in
Appendix~\ref{app:hsing}. It is convenient to express the results in
terms of the quantity
\begin{equation}
  \delta:=\sqrt{3(1-y)}
\end{equation}
which is such that $y=1-\delta^2/3$ and $x(y)=1-\delta+O(\delta^2)$.

We start by the case where $t$ is fixed: it is important to note that
the limit depends on the angle at which $\delta$ approaches $0$. More
precisely, we fix an arbitrary $\alpha \in (-\frac\pi2,\frac\pi2)$ and assume that
$\arg \delta \to \alpha$, i.e.\ $\arg(1-y) \to 2\alpha$.
Note that we do not consider the situation where $y$ approaches $1$
tangentially to the half-line $[1,+\infty)$.  
\begin{prop}
  Assume that $t$ does not belong to the curve
  \label{prop:hexptfin}
  $\mathcal S_\alpha:=\{ e^{e^{i\alpha} v}, v > 0\}$.
  Then, as $\delta \to 0$ with $\arg\delta \to \alpha$, we have
  \begin{equation}
    \label{eq:hexptfin}
     \hat h\left(t,z(1-\delta^2/3)\right) = H_\alpha(t) + O(\delta)
  \end{equation}
  where
  \begin{equation}
    \label{eq:Halphadef}
    H_\alpha(t) := \frac94 + 9(t-1) \Phi_\alpha(t), \qquad
    \Phi_\alpha(t) := \int_{e^{i\alpha} \R_+} dw \frac{w e^{-3 w}}{1-t e^{-w}}.
  \end{equation}
  The error term $O(\delta)$ above is uniform in $t$ within any
  compact subset of $\C \setminus \mathcal S_\alpha$. Furthermore,
  under the assumption that $\arg \delta=\alpha+O(\delta)$, the
  approximation still holds for $t$ approaching $1$ in such a way that
  $|t-1| \gg |\delta|$ and that $\arg(t-1)$ remains away from
  $\alpha$.
\end{prop}

\begin{rem} \label{rem:Htexp} The functions $\Phi_\alpha(t)$ and
  $H_\alpha(t)$ are analytic continuations to the domain
  $\C \setminus \mathcal S_\alpha$ of the respective series
  \begin{equation}
    \label{eq:Htexp}
    \Phi(t) := \sum_{n=0}^\infty \frac{t^n}{(n+3)^2} = \frac{\Li_2(t)-t-\frac{t^2}4}{t^3}, \qquad
    H(t) := \sum_{\ell=1}^\infty \frac{9(2\ell+3)}{(\ell+1)^2 (\ell+2)^2} t^{\ell-1}
      \end{equation}
  which have radius of convergence $1$ (here $\Li_2$ is the
  dilogarithm).  The coefficient of $t^{\ell-1}$ in $H(t)$ is nothing
  but the limiting value of $h_\ell(z)$ at $z=z(1)=\frac{4}{27}$, as
  may be seen from \eqref{eq:klFromMinbus} (take $r=3/2$ and
  $x \to 1$). Deducing that $\hat h(t,z)$ tends to $H(t)$ for $|t|<1$
  requires a justification for the exchange of limit and summation,
  given in Appendix~\ref{app:hsing}. We note finally that, since all
  the coefficients of $H(t)$ are positive, we have
  \begin{equation}
    |H(t)| < H(1)=\frac94, \qquad \text{if } |t|\leq 1 \text{ and } t\neq 1
    \label{eq:Daffodil_crit}
  \end{equation}
  as seen for instance by the Daffodil Lemma~\cite[Lemma
  IV.1]{Flajolet2009}.
\end{rem}

We then consider the case where $t$ is at a distance of order $\delta$
from $1$:
\begin{prop}
  Given $\tau \in \C \setminus \Z_{>0}$, we have as $\delta \to 0$ 
  \label{prop:hexptsmall}
  \begin{multline}
    \label{eq:hexptsmall}
  \hat h\left(1+\tau \delta,z(1-\delta^2/3)\right) = \frac{9}{4} + \frac34 (2\pi^2-15) \tau \delta - 9 \tau^2 \delta^2 \ln \delta \\
  -9 \left( \tau^2 \psi(-\tau) + \frac{2\pi^2-13}{4} \tau^2 - \frac\tau2 + \frac1{12} \right) \delta^2 + O\left(\delta^{7/3}\right)
\end{multline}
where $\psi=\frac{\Gamma'}{\Gamma}$ is the digamma function, the
logarithmic derivative of the gamma function. The error term
$O(\delta^{7/3})$ is uniform for $\tau$ in any compact subset of
$\C \setminus \Z_{>0}$. It furthermore remains negligible for $\tau$
mildly large in the following sense: for any $b \in (0,\frac15)$ and
$\vr>0$, the above expansion holds with an error term
$O(\delta^{(7-5b)/3})$ uniform over all $\tau$ such that
$|\tau| \leq \delta^{-1/b}$ and $|\tau-k| \leq \vr$ for all $k \in \Z_{>0}$.
\end{prop}

It is worthwhile to note that the two expansions above are consistent
with one another. Indeed, the function $H_\alpha(t)$ satisfies for
$t=1+\tau \delta$
\begin{equation}
  \label{eq:hphiexp}
  H_\alpha(t) = \frac94 + \frac34 (2\pi^2-15) \tau \delta
  - 9 \tau^2 \delta^2 \ln (-\tau \delta)
  -\frac94 (2\pi^2-13) \tau^2 \delta^2 + o(\delta^2),
\end{equation}
with the cut of $\ln$ depending appropriately on $\alpha$. This
expansion may be recovered from~\eqref{eq:hexptsmall} by taking
$\tau \to \infty$ keeping $\tau \delta$ constant, and using
$\psi(-\tau) \sim \ln (-\tau)$. We note that the difference
between~\eqref{eq:hexptsmall} and \eqref{eq:hphiexp} reads
\begin{equation}
  \hat h(t,z(y)) - H_\alpha(t) = \left( 9 \tau^2 \left(\ln(-\tau) - \psi(-\tau)\right) + \frac{9\tau}{2} - \frac34 \right) \delta^2 + o(\delta^2) 
\end{equation}
which matches~\cite[Equation~(3.28)]{quadwithnoME} that reads, in our
present notations,
\begin{equation}
  \hat h\left( t,z(y) \right) - H_\alpha(t)
  \sim \delta^2 \int_{0}^{\infty}dL e^{L \tau}
  \left( \frac{9}{4} \cdot
    \frac{\cosh (L/2)}{\sinh^3 (L/2)} - \frac{18}{L^3} \right), \qquad
  \Re(\tau)<0.
\end{equation}

We complete our analysis of the singular behavior of $\hat h$ by
observing that, upon taking $\tau=\tau' \delta$ in
Proposition~\ref{prop:hexptsmall} (which is possible thanks to the
uniformity of the error), we get the following corollary for $t$ even
closer to $1$, namely at a distance of order $\delta^2$ i.e.\ $|y-1|$:
\begin{cor}
  \label{cor:hexptverysmall}
  Assume that $t-1 \sim \tau' \delta^2$, with $\tau' \in \C$. Then, for
  $\delta \to 0$ we have
  \begin{equation}
    \label{eq:hexptverysmall}
    \hat h(t,z(y)) = \frac{9}{4} + \frac34 \left((2\pi^2-15) \tau' -1 \right) \delta^2 + O\left(\delta^{7/3}\right)
  \end{equation}
  where the error is uniform for $\tau'$ in any compact subset of
  $\C$.
\end{cor}

Plugging the above approximations into~\eqref{eq:Ktgyu} provides in
turn useful approximations for $\hat K(t,g(y),u)$. Namely, noting that
$z(1)/g(1)=(1+u/3)^2$, Proposition~\ref{prop:hexptfin} implies
immediately that, as $\delta \to 0$ with $\arg\delta \to \alpha$, we
have
\begin{equation}
  \label{eq:Kexptfin}
  \hat K(t,g(y),u) = \left(1 + \frac{u}3 \right)^2 \frac{H_\alpha(t)-1}{u+(1-u) H_\alpha(t)} + O(\delta)
\end{equation}
for any
$t \in \C \setminus \mathcal S_\alpha \setminus
H_\alpha^{-1}(\{\frac{u}{u-1}\})$. This holds in particular for any
$t$ in the open unit disk, where $H_\alpha(t)=H(t)$ is independent of
$\alpha$ and satisfies $\Re H(t) \in (\frac 14,\frac 94)$ --- as seen by
exploiting again the positivity of the coefficients in~\eqref{eq:Htexp}
to write
$|H(t)-\frac 54| \le \sum_{\ell=2}^\infty \frac{9(2\ell+3)}{(\ell+1)^2 (\ell+2)^2} =1$
 --- while we have
$\frac{u}{u-1} \notin (\frac 14,\frac 94)$ for $u \in
[0,\frac95]$. By~\eqref{eq:Kellgyu} and by the uniformity of the error
term in a neighborhood of the origin, we deduce that, for any integer
$\ell$, $K_\ell(g(y),u)$ has a finite limit for $y \to 1$, given by
the coefficient of $t^{\ell-1}$ in the expansion of the right-hand
side of~\eqref{eq:Kexptfin} around $t=0$. This justifies that we may
take, in the integral representation~\eqref{eq:Klnuy_integraldef} for
$K_\ell^{(n)}$, the \emph{same} contours as those considered in
Section~\ref{sec:asymptQ}, which pass through the saddle point $\ys=1$.

We conclude this subsection by observing that, for $t$ close to $1$,
plugging~\eqref{eq:hexptsmall} or~\eqref{eq:hexptverysmall}
into~\eqref{eq:Ktgyu} gives rise to approximations for
$\hat K(t,g(y),u)$ that differ completely between the critical
($u=9/5$) and subcritical ($u<9/5$) cases. Indeed, since
$\hat h(t,z(y))$ gets close to $9/4$, we find that
$\hat K(t,g(y),u)$ diverges in the former case, and remains finite in
the latter case. We explore these two respective situations in the following subsections.

\subsection{The critical case}
\label{ssec:critKl}

Let us discuss the critical case $u = 9/5$. From
Corollary~\ref{cor:hexptverysmall} we find that, in the regime
$t-1 \sim \tau' \delta^2$, we have
\begin{equation}
  \label{eq:Kexptverysmall}
  \hat K(t,g(y),9/5) = - \frac{16}{3(2\pi^2 - 15)} \cdot \frac1{\tau' - \frac{1}{2\pi^2-15}} \cdot \delta^{-2} + O(\delta^{-1}).
\end{equation}
Intuitively speaking, this means that $\hat K(t,g(y),9/5)$ has a pole
at $t=1+\frac{\delta^2}{2\pi^2-15}+O(\delta^3)$, which is very
analogous to the situation described by Lemma~\ref{lem:t0yexpansuper}
in the supercritical regime (in some sense this lemma remains valid
for $\xs=1$). The situation is however here more complex, as $\hat K$
has other poles near $t=1$ which are
induced\footnote{From~\eqref{eq:hexptsmall} we see heuristically that
  the equation $\hat h=\frac94$ has a solution at
  $\tau=k+\frac{12 k \delta}{2\pi^2-15}+o(\delta)$ for every
  $k \in \Z_{>0}$: because of its pole the digamma term
  $-9 \tau^2 \psi(-\tau) \delta^2$ is in fact of order $\delta$ and
  compensates precisely the other first-order term. In the variable
  $t$ such solution is at a distance of order $\delta$ from $1$, hence
  much further than the dominant solution which is at a distance of
  order $\delta^2$.} by those of $\hat h$. We will show below that
this does not affect the asymptotics of $K_\ell$ at leading order. We
may again read off informally the relevant scaling regime for $\ell$:
since the dominant pole in $t$ of $\hat K$ is at a distance $\delta^2$
from $1$, if we want the corresponding residue in~\eqref{eq:Kellgyu}
to vary, we should take $\ell$ of order
$\delta^{-2} \propto |y-1|^{-1}$, which should itself scale as
$n^{1/3}$ from the discussion of Section~\ref{sssec:critQ}. Precisely,
we have the following asymptotics.

\begin{prop}    \label{prop:K_ellnasymptcrit}
  For any $d>0$, we have as $n \to \infty$
  \begin{equation}
    K_{\lfloor d n^{1/3} \rfloor}^{(n)} \sim \frac{2^{10/3}}{9 \Gamma(1/3)} \left( \frac{432}{25} \right)^n n^{-1}
    \times 3^{1/3} \Gamma(1/3) \frac{d}{ D^2} \Airy \left(\frac{d}{D} \right)
    \label{eq:K_ellnasymptcrit}
  \end{equation}
  with $\Airy$ the Airy function, and $D := \frac{2 \pi^2 - 15}{6^{2/3}}$.
\end{prop}

Let us prove this proposition, and set
$\ell := \lfloor d n^{1/3} \rfloor$ from now on.  In the integral
representation~\eqref{eq:Klnuy_integraldef} for $K_\ell^{(n)}$, we
choose the contour in the variable $y$ to be the portion of Ceva
trisectrix $y(\theta)$ considered in Section~\ref{sssec:critQ}.
Splitting the interval for $\theta$ into the same three parts as
before (see again Figure~\ref{fig:CevaTrixArcs}), we denote their
corresponding contributions to $K_\ell^{(n)}$ by $K_{\ell;+}^{(n)}$,
$K_{\ell;-}^{(n)}$ and $K_{\ell;\mathrm{tail}}^{(n)}$ respectively.

\paragraph{Central approximation.} Let us first note that, by symmetry, we have
\begin{equation}
  K_{\ell;+}^{(n)} = (K_{\ell;-}^{(n)})^*.
\end{equation}
Thus, it suffices to study the contribution
$K_{\ell;-}^{(n)}$ of the half of central arc corresponding to
$\theta=-\frac\pi3+s n^{-1/3}$, $s \in [0,n^{\epsilon}]$ with $\epsilon=1/24$.
We have the following key estimate for $K_{\ell}(g(y),9/5)$ on this arc.

\begin{lem}
  For $y=y(-\frac{\pi}{3} + s n^{-1/3})$, $s \in [0,n^{\epsilon}]$, and $\ell=\lfloor d n^{1/3} \rfloor$, we have
  \begin{equation}
\!\!  K_{\ell}(g(y),9/5)
    = \frac{16}{3(2\pi^2 - 15)} e^{-\frac{3\sqrt{3}}{2\pi^2 - 15} e^{-i \pi/3} s d} 
    \!+ O\!\left(\sqrt{s}n^{-1/6}\right)
    \!+ O\!\left(\frac{n^{1/6}}{\sqrt{s}}e^{-c d n^{1/6} \sqrt{s}}\right)
    \label{eq:critTPF_centralApprox}
  \end{equation}
  for some $c>0$. 
  \label{lem:critTPF_centralApprox}
\end{lem}

\begin{proof}
  Recall the notation $\delta := \sqrt{3(1-y)}$.  We start from the
  integral representation~\eqref{eq:Kellgyu} for $K_{\ell}(g(y),9/5)$,
  and deform the contour into the union of the two following circles:
  \begin{itemize}
  \item a small negatively oriented circle $\mathcal C_1$ with center $1$ and radius $|\delta|^2$,
  \item a large positively oriented circle $\mathcal C_2$ with
    center $0$ and radius $e^{|\delta|/2}$.
  \end{itemize}
  We claim that, for $\delta$ small enough, this deformation is valid,
  namely that the region $\mathcal R$ consisting of the interior of
  $\mathcal C_2$ deprived from that of $\mathcal C_1$ contains no pole
  of $t \mapsto \hat K(t,g(y),9/5)$, i.e.\ no solution of
  $\hat h(t,z(y))=\frac94$. To this end, we use the approximations of
  the previous subsection. First, we observe that there exists $a>0$
  such that
  \begin{equation}
    \label{eq:Hineq}
    \left\lvert H_\alpha(t)-\frac94 \right\rvert \geq a |t-1|, \qquad \forall t \in \mathcal R.
  \end{equation}
  Indeed this is the case in a small fixed neighborhood of $t=1$ by
  the expansion~\eqref{eq:hphiexp}, and outside this neighborhood we
  use~\eqref{eq:Daffodil_crit} and note that this inequality extends
  to $|t|\leq e^{|\delta|/2}$, for $\delta$ small enough, by analytic
  continuation. Now, to prove that $\hat h(t,z(y)) \neq \frac94$ for
  $t \in \mathcal R$, we distinguish the following several cases.
  \begin{itemize}
  \item For $|t-1| \gg |\delta|$, we combine~\eqref{eq:Hineq} and
    Proposition~\ref{prop:hexptfin}, noting that $\arg(t-1)$ remains
    away from $\alpha$ for $t \in \mathcal R$.
  \item For $t = 1+O(\delta)$ and $|t-1| \gg |\delta|^2$, we see that
    $h(t,z(y))-9/4$ is small but non-zero, as the first order
    correction in \eqref{eq:hexptsmall} does not vanish, and as
    $\tau$ does not approach integer values for $t$ in
    $\mathcal R$.
  \item For $t = 1 + O(\delta^2)$ we again see that $h(t,z(y))-9/4$ is
    small but non-zero, using now \eqref{eq:hexptverysmall} and noting
    that $2\pi^2 - 15 > 1$ to see that the correction is non-zero
    outside $\mathcal C_1$, which ensures $|\tau'| \geq 1$.
  \end{itemize}
  This shows that the deformation of the contour is valid, and hence
  we have
  \begin{equation}
    K_\ell(g(y),u) = \frac{1}{2i\pi} \oint_{\mathcal C_1} \frac{\hat K(t,g(y),9/5)dt}{t^\ell}
     + \frac{1}{2i\pi} \oint_{\mathcal C_2} \frac{\hat K(t,g(y),9/5)dt}{t^\ell}.
   \end{equation}
   We treat the integral over $\mathcal C_1$ by doing the change of
   variable $t \to \tau'=\frac{t-1}{\delta^2}$ and
   using~\eqref{eq:Kexptverysmall}: recalling that $\mathcal C_1$ is
   negatively oriented and applying the residue theorem yields
   \begin{equation}
     \frac{1}{2i\pi} \oint_{\mathcal C_1} \frac{\hat K(t,g(y),9/5)dt}{t^\ell} =
     \frac{16}{3(2\pi^2 - 15)} \left(1+\frac{\delta^2}{2\pi^2 - 15}\right)^{-\ell} + O(\delta)
   \end{equation}
   which, using $\delta^2=3(1-y)=3\sqrt{3}e^{-i\pi/3}s\, n^{-1/3}+O(s^2n^{-2/3})$, yields the first two terms in the right-hand side
   of~\eqref{eq:critTPF_centralApprox}. As for the integral over
   $\mathcal C_2$, it yields the third (error) term, since along
   $\mathcal C_2$ we have $|t|^\ell \sim e^{-c d n^{1/6} \sqrt{s}}$
   and the modulus of $\hat K(t,g(y),9/5)$ is a
   $O(\delta^{-1})=O(n^{1/6}/\sqrt{s})$ by the approximation results
   of the previous subsection. Note that this error term potentially
   blows up for $s \to 0$, but we will see below that it yields a negligible contribution
   once integrated over $s$.
\end{proof}

Now, to estimate $K_{\ell;-}^{(n)}$, we shall multiply 
the estimate \eqref{eq:critTPF_centralApprox} for $K_{\ell}(g(y),9/5)$
by that \eqref{eq:gyncentralcrit} for $g(y)^{-n}$, and, from \eqref{eq:approxlog_g_crit},
by the factor
\begin{equation}
  \frac{g'(y)}{g(y)} \cdot \frac{dy}{ds} = \frac{d \log g(y)}{ds} = \frac{9 \sqrt{3} s^2}{4 n}\left( 1 + O(n^{\epsilon-1/3}) \right),
  \label{eq:dlog_g_crit_ceva}
\end{equation}
to yield
\begin{multline}
  K_{\ell;-}^{(n)} = \frac{1}{2 i \pi}
    \frac{16}{3 (2 \pi^2 - 15)}
    \frac{9 \sqrt{3}}{4 n} 
    \left( \frac{432}{25} \right)^{\!\! n}
    \int_{0}^{\infty} ds \, s^2 
    e^{- \frac{3\sqrt{3}}{4} s^3 - \frac{3\sqrt{3}}{2\pi^2-15} e^{- i \pi /3} s d}
    \times \\
    \Bigg( 1+ O(n^{\epsilon-1/6}) + \underbrace{O\left( 
          n^{1/6} \int_0^{n^\epsilon} ds \, s^{3/2} e^{-c d n^{1/6} \sqrt{s}}
    \right)}_{O(n^{-2/3})} \Bigg).
\end{multline}
Combining with the symmetric contribution $K_{\ell;+}^{(n)}$,
and using the integral evaluation
\begin{equation}
  \Re\!\left[\frac{1}{2 i \pi} \int_{0}^{\infty}\!\!\! ds \, s^2 e^{- \frac{3\sqrt{3}}{4} s^3 - \frac{3\sqrt{3}}{2\pi^2-15} e^{- i \pi /3} s d} \right] = \frac{2^{5/3}}{3^{11/6}(2 \pi^2 - 15)} d \Airy\left(\frac{6^{2/3}}{2 \pi^2 - 15} d \right)
\end{equation}
we find that $K_{\ell;+}^{(n)}+K_{\ell;-}^{(n)}$ is asymptotically equivalent to the right-hand side of~\eqref{eq:K_ellnasymptcrit}.

\paragraph{Tail bound.} We now complete the proof of
Proposition~\ref{prop:K_ellnasymptcrit} by showing that the
contribution $K_{\ell;\mathrm{tail}}^{(n)}$ of the tail arc is
negligible. We proceed as in the proof of
Proposition~\ref{prop:K_ellnasympt} by splitting it into \emph{near}
and \emph{far} tail regions. More precisely, we note that the
statement and proof of Lemma~\ref{lem:critTPF_centralApprox} remain
valid if we take $s \in [n^{\epsilon},a n^{1/3}]$ for some small
$a>0$. On this interval, corresponding to a half of the near tail
region, $K_{\ell}(g(y),9/5)$ remains bounded, and since we multiply
this quantity by $g(y)^{-n}=O(g(1)^{-n} e^{-c' n^{3\epsilon}})$, we
get a negligible contribution. By symmetry the same conclusion holds
for the other half of the near tail region, corresponding to
$\theta=\frac{\pi}3-s n^{-1/3}$ with $s \in [n^{\epsilon},a
n^{1/3}]$. Finally, the far tail region corresponds to
$\theta \in [-\frac\pi3+a,\frac\pi3-a]$ and we use the same reasoning
as in the proof of Proposition~\ref{prop:K_ellnasympt} to show that it
gives rise to a negligible contribution (the main modification is that
$\ell$ is of order $n^{1/3}$ instead of $n^{1/2}$, which does not
affect the conclusion). Proposition~\ref{prop:K_ellnasymptcrit} is now
established.

Equation  \eqref{eq:distanceprofile_crit} is then obtained by using
$K^{(n)}(9/5) \sim 2n Q_n$ with $Q_n$ as in \eqref{eq:Qn_asymcrit}.

\subsection{The subcritical case}
\label{ssec:subcritKl}

Let us now turn to the subcritical case $u<9/5$. As mentioned at the
end of Section~\ref{ssec:hsingtext}, $\hat K(t,g(y),u)$ now remains
finite for $y$ and $t$ close to $1$. More precisely it has,
by~\eqref{eq:Ktgyu}, an expansion for $t-1$ of order $\delta$ that is
very similar to~\eqref{eq:hexptsmall}, where the different terms get
multiplied by nonvanishing rational functions of $u$ (given explicitly
below). From this, we expect that $K_\ell(g(y),u)$ and $h_\ell(z(y))$,
and in turn $K_\ell^{(n)}$ and $h_\ell^{(n)}:=[z^n] h_\ell(z)$, will
have similar asymptotics. The relevant scaling regime is here known
from~\cite{quadwithnoME} and can be read informally
from~\eqref{eq:klFromMinbus}: as we take
$x=x(y)=e^{-\delta+O(\delta^2)}$, we should take $\ell$ of order
$\delta^{-1} \propto |y-1|^{-1/2}$, which should itself scale as
$n^{1/4}$ from the discussion of Section~\ref{sssec:noncritQ}. The
precise asymptotics are the following.

\begin{prop}
  For any $d>0$, we have as $n \to \infty$
  \begin{equation}
    \label{eq:Knsub}
    K^{(n)}_{\lfloor d n^{1/4} \rfloor} \sim \frac{4}{\sqrt{\pi}} \left( \frac{3+u}{9-5u} \right)^{5/2} \left( \frac{3(3+u)^2}4 \right)^n 
    \frac{\rho_u(d)}{n^{7/4}} 
  \end{equation}
  with $\rho_u(d)$ as in~\eqref{eq:distanceprofile_subcrit}.
\end{prop}

Let us prove this proposition, and set
$\ell := \lfloor d n^{1/4} \rfloor$ from now on.  In the integral
representation~\eqref{eq:Klnuy_integraldef} for $K_\ell^{(n)}$, we
choose the contour in the variable $y$ to be the circle of radius
$\ys=1$ as in Section~\ref{sssec:noncritQ}.  Splitting the circle in two
parts as before (see again Figure~\ref{fig:circleArcs}), we denote
their corresponding contributions to $K_\ell^{(n)}$ by
$K_{\ell;\mathrm{central}}^{(n)}$ and $K_{\ell;\mathrm{tail}}^{(n)}$
respectively.

Recall that the central arc
consists of the points whose argument belongs to
$[-n^\epsilon,n^\epsilon]$, with $\epsilon > 0$ to be fixed later. On this arc we have
the following key estimate for $K_\ell(g(y),u)$:

\begin{lem}
  For $y = e^{i s n^{-1/2}}$, $s \in [-n^\epsilon,n^\epsilon]$, and
  \label{lem:subcritTPF}
  $\ell=\lfloor d n^{1/4} \rfloor$, we have
  \begin{equation}
    K_\ell(g(y),u) = 4 \left( \frac{3+u}{9-5u} \right)^{2} 
    \frac{\cosh \!\left( \frac{d}{2} \sqrt{-3 i s} \right)}{\sinh^3 \!\left( \frac{d}{2} \sqrt{-3 i s} \right)}
    \left( \frac{-3 i s}{\sqrt{n}} \right)^{3/2} + O\left( n^{\frac{19}{12}(\epsilon - \frac{1}{2})} \right).
    \label{eq:Kellofy_approxn_subcrit}
  \end{equation}
\end{lem}

\begin{proof}
  From its explicit expression \eqref{eq:klFromMinbus},
  with $x$ and $r$ given by \eqref{eq:xyexplicit} and \eqref{eq:ryexp},
  we have for small $\delta$,
  for $y = 1 - \frac{\delta^2}{3}$ and $\ell := \lfloor \frac{L}{\delta} \rfloor$:
  \begin{equation}
    h_{\ell}(z(y)) = \frac{9}{4} \cdot
    \frac{\cosh \!\left( \frac{L}{2} \right)}{\sinh^3 \!\left( \frac{L}{2} \right)}
  \delta^{3} + O\left( \delta^{4} \right).
    \label{eq:hlsubcrit}
  \end{equation}
  
  Let us show that $K_{\ell}(g(y),u)$ is at leading order
  proportional to $h_{\ell}(z(y))$, more precisely:
  \begin{equation}
    K_{\ell}(g(y),u) - c(u) \cdot h_{\ell}(z(y)) = O\left(\delta^{19/6}\right), \qquad c(u) := \frac{16}{9} \left( \frac{3+u}{9-5u} \right)^{2}.
    \label{eq:Kl_hl_prop_relation}
  \end{equation}
  To this end, we write the difference as a contour integral
  \begin{equation}
    K_{\ell}(g(y),u) - c(u) \cdot h_{\ell}(z(y)) = \frac{1}{2i \pi} \oint \frac{\hat K(t,g(y),u) - c(u) \cdot \hat h(t, z(y))}{t^{\ell}} dt.
    \label{eq:hatK_hath_contour_integral}
  \end{equation}
  where the contour of integration, initially a small circle around the origin,
  is deformed into the union of the following contours:
  \begin{itemize}
    \item a collection of $T-1$ small negatively oriented circles
      where $T:=\lfloor \delta^{-b} \rfloor$
      and $b$ a small positive number to be fixed later,
      all with the same radius $\left|\frac{\delta}{2}\right|$,
      and with centers $1 + k \delta$ for $k = 1, \ldots, T-1$.
    \item a large positively oriented circle with center $0$
      and passing through $1+\left( T-\frac{1}{2} \right)\delta$.
  \end{itemize}
  The proof that this deformation is valid works similarly
  to the case of Lemma~\ref{lem:critTPF_centralApprox}:
  we check, using results from Section~\ref{ssec:hsingtext},
  that the poles of the integrand\footnote{except, of course,
  the pole at $t=0$ which we are analyzing.}
  stay outside of that contour.
  Indeed, the poles of $\hat K(t,g(y),u)$
  and of $\hat h(t,z(y))$ both stay contained in the small circles,
  one of each within each circle.

  We now bound the contribution of each circle in the following way.
  For each small circle, we perform a change of variable $t = 1 + \tau \delta$.
  By using \eqref{eq:Ktgyu} and Proposition~\ref{prop:hexptsmall} (assuming $b<1/5$),
  we get
  \begin{equation}
    \hat K(t,g(y),u) - c(u) \cdot \hat h(t, z(y)) = c(u) \left( 
    \frac{9-25 u}{16} + \frac{9(u-1)(2 \pi^2 - 15)^2}{4(9-5u)}\tau^2 \delta^2 + O(\delta^{\frac{7-5b}{3}}) \right).
    \label{eq:hatK_hath_expension}
  \end{equation}
  Multiplying by $t^{-\ell} = (1+\tau\delta)^{-\ell} = O(1)$,
  we see that the non-error terms in the above expansion are regular
  and yield a zero contribution to the contour integral,
  while the error term yields an error $O(\delta^{\frac{10-5b}{3}})$ per small circle,
  hence a total error $O(T\delta^{\frac{10-5b}{3}}) = O(\delta^{\frac{10-8b}{3}})$.
  To get an error term smaller than the leading order $\delta^3$ in \eqref{eq:hlsubcrit},
  we fix $b = 1/16 < 1/8$. 
  As for the large circle, we note that the choice of its radius
  ensures we stay away from the poles of the integrand: as in the proof of Lemma~\ref{lem:critTPF_centralApprox},
  we use Proposition~\ref{prop:hexptfin} for $t$ away from $1$,
  and Proposition~\ref{prop:hexptsmall}  for $t$ close to $1$. 
  The integral over the large circle is then bounded by a constant times
  \begin{equation}
    \frac{1}{\left| 1+(T-1/2)\delta \right|^{\ell}} \sim
  \frac{1}{\left| e^{(T-1/2)\delta\ell} \right|} \sim
  e^{-(T-1/2) \Re(\delta)\ell} = O\left(e^{-a |\delta|^{-1/16}}\right)
    \label{eq:todo}
  \end{equation}
  for some $a > 0$.
  Altogether, we get the estimate \eqref{eq:Kl_hl_prop_relation}.
  Combining with \eqref{eq:hlsubcrit} and setting $\delta = \sqrt{3(1-y)} = \sqrt{\frac{-3is}{\sqrt{n}}}\left(1 + O(\frac{s}{\sqrt{n}})\right)$, we get the wanted result.
\end{proof}

Using \eqref{eq:gyncentral} with $g(\ys)$ given by \eqref{eq:gysaddle}, and \eqref{eq:gy_derivlog_trick},
we get from \eqref{eq:Klnuy_integraldef}:
\begin{equation}
  \begin{split}
     K_{\ell;\mathrm{central}}^{(n)} = \left( \frac{3(3+u)^{2}}{4} \right)^{n} & \frac{1}{2 i \, \pi}
    \int_{-\infty}^{+\infty} ds \frac{2 c s}{n} e^{- c s^2}
    \left( \frac{-3 i s}{\sqrt{n}} \right)^{3/2} \\
    &\times \left\{ \left( \frac{3+u}{9-5u} \right)^{2} 4
    \frac{\cosh(\frac{d}{2} \sqrt{-3 i s})}{\sinh^3(\frac{d}{2} \sqrt{-3 i s})}
    + O( n^{\frac{19}{12}(\epsilon - \frac{1}{2})}) \right\}
  \end{split}
\end{equation}
with $c=-\frac{1^2 g''(1)}{2g(1)}=\frac{9-5u}{4(3+u)}$, that is,
taking $\epsilon < 1/2$:
\begin{equation}
  \begin{split}
    K_{\ell;\mathrm{central}}^{(n)} &\sim \left( \frac{3(3+u)^{2}}{4} \right)^{n} \frac{1}{2 i \, \pi}
    \int_{-\infty}^{+\infty}\hspace{-1em}  ds \frac{2 c s}{(4c)^2 n} e^{- c s^2} 
    \left( \frac{-3 i s}{\sqrt{n}} \right)^{\!3/2} \hspace{-.5em} 4
    \frac{\cosh(\frac{d}{2} \sqrt{-3 i s})}{\sinh^3(\frac{d}{2} \sqrt{-3 i s})} \\
    &= \frac{4}{\sqrt{\pi}} \left( \frac{3+u}{9-5u} \right)^{5/2} \left( \frac{3(3+u)^2}4 \right)^n 
    \frac{1}{n^{7/4}} \tilde\rho_u(d)
  \end{split}
  \label{eq:subcritKcentral_approx}
\end{equation}
with 
\begin{equation}
    \tilde\rho_u(d) :=  
    \frac{(4 c)^{5/2}}{8 i \sqrt{\pi}}
    \int_{-\infty}^{+\infty} ds \frac{2 s}{c} e^{- c s^2}
    \frac{1}{4}
    \frac{\cosh(\frac{d}{2} \sqrt{-3 i s})}{\sinh^3(\frac{d}{2} \sqrt{-3 i s})}
    \left( -3 i s \right)^{3/2} 
    =  \Phi_u'(d) 
\end{equation}
where
\begin{equation}
  \Phi_u(d) := \frac{2 c^{3/2}}{\sqrt{\pi}}
    \int_{-\infty}^{+\infty} s^2 e^{- c s^2}
    \left( 1 + \frac{3}{\sinh^2(\frac{d}{2} \sqrt{-3 i s})} \right)ds.
\end{equation}
Here the term $1$ in the integrand was chosen so that
$\Phi_u(d) \overset{d\to 0}{\longrightarrow} 0$.
We then have $\Phi_u(d) \overset{d\to \infty}{\longrightarrow} 1$,
which ensures that $\tilde\rho_u$ is indeed a well-normalized probability density.

\bigskip

Using \eqref{eq:subcritKcentral_approx}, arguing as in the previous subsection that the tail arc
yields a negligible contribution, and using as before 
$K^{(n)} \sim 2n Q_n$ with $Q_n$ as in \eqref{eq:Qn_asymsub},
we see that
\begin{equation}
  \rho_u(d) = \lim_{n\to \infty} n^{1/4} \frac{K_{\lfloor d n^{1/4} \rfloor}^{(n)} (u)}{K^{(n)} (u)} = \tilde\rho_u(d)
\end{equation}
which identifies $\tilde\rho_u(d)$ as the distance profile.
The dependency in $u$ is best displayed by writing
\begin{equation}
  \rho_u(d) = \frac{1}{D} \  \rho_{\mathrm{sub}}\!\left( \frac{d}{D} \right),
  \qquad D := (4c)^{1/4} = \left( \frac{9-5u}{3+u} \right)^{1/4}
  \label{eq:rho_scaling_subcrit}
\end{equation}
where $\rho_{\mathrm{sub}}(x) := \Phi'(x)$ with
\begin{equation}
  \begin{split}
    \Phi(x) &:= \frac{2}{\sqrt{\pi}}
    \int_{-\infty}^{+\infty} s^2 e^{- s^2}
    \left( 1 + \frac{3}{\sinh^2\left(\frac{x \sqrt{2}}{2} \sqrt{-3 i s} \right)} \right)ds \\
    &= \frac{4}{\sqrt{\pi}}
    \int_{0}^{+\infty} s^2 e^{- s^2}
    \left( 1 - 6 \frac{1 - \cosh\left( x\sqrt{3 s} \right) \cos\left( x\sqrt{3 s} \right)}{\left( \cosh\left( x\sqrt{3 s} \right) - \cos\left( x\sqrt{3 s} \right) \right)^2} \right)ds
  \end{split}
\end{equation}
or equivalently, integrating by parts
\begin{equation}
  \rho_{\mathrm{sub}}(x) = \frac{8}{x \sqrt{\pi}}
  \int_{0}^{+\infty} s^2 (2 s^2 - 3) e^{- s^2}
  \left( 1 - 6 \frac{1 - \cosh\left( x\sqrt{3 s} \right) \cos\left( x\sqrt{3 s} \right)}{\left( \cosh\left( x\sqrt{3 s} \right) - \cos\left( x\sqrt{3 s} \right) \right)^2} \right)ds.
\end{equation}
This is the result announced in \eqref{eq:distanceprofile_subcrit}.

\subsection{Block two-point function in the critical case}
\label{ssec:rootblockcrit}

In this section, we compute the \emph{block distance profile} $\rho_{\mathrm{block}}(d)$
at the critical point $u = 9/5$.
It is defined in terms of 
the block two-point function $L_{\ell}(g,u)$ introduced in \eqref{eq:substLl} as follows.
Introducing the quantities $L_{\ell}^{(n)} := [g^n]L_{\ell}(g,9/5)$ and $L^{(n)} := \sum_{\ell \geq 1} L_{\ell}^{(n)} = [g^n]\hat L(1,g,9/5)$, 
we define
\begin{equation}
  \rho_{\mathrm{block}}(d) := \lim_{n \to \infty} \frac{n^{1/6} L_{\lfloor d n^{1/6}\rfloor}^{(n)}}{L^{(n)}}.
  \label{eq:rhoblock_def}
\end{equation}
\begin{prop}
  We have:
  \begin{equation}
    \rho_{\mathrm{block}}(d) = 
    \frac{2^{13/3}\ 3}{\Gamma(7/3)} \int_0^{\infty} d\mu \ \mu^{7/2} e^{-\mu^{3}}
    \left\{ \breve s \frac{\tilde c(\tilde c+\breve c) -2}{(\tilde c-\breve{c})^3} \right\}
    \label{eq:rhoblock_d_real}
  \end{equation}
  (see Figure~\ref{fig:distance_profile_Ll}), with the shorthand notations
  \begin{equation}
    \tilde c = \cos\left(\frac{\sqrt{3\mu}}{2^{2/3}}d\right), \quad \tilde s = \sin\left(\frac{\sqrt{3\mu}}{2^{2/3}}d\right),
    \quad \breve c = \cosh\left(\frac{3\sqrt{\mu}}{2^{2/3}}d\right), \quad \breve s = \sinh\left(\frac{3\sqrt{\mu}}{2^{2/3}}d\right).
  \end{equation}
  \label{prop:L_ellnasymptcrit}
\end{prop}
\begin{figure}[h]
  \centering
  \includegraphics[width=.55\textwidth]{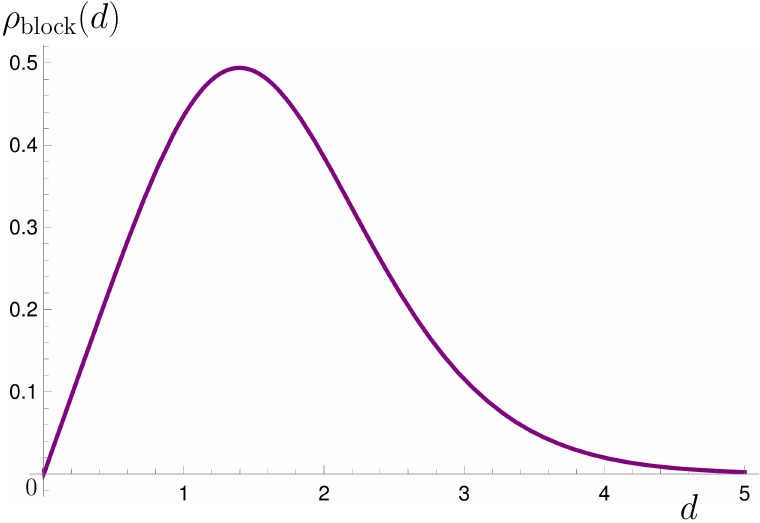}
  \caption{The block distance profile $\rho_{\mathrm{block}}(d)$ at the critical value $u=9/5$.}
  \label{fig:distance_profile_Ll}
\end{figure}

We use the following analog of \eqref{eq:Klnuy_integraldef}:
\begin{equation}
  L_\ell^{(n)} = \frac{1}{2i\pi} \oint \frac{L_\ell(g(y),9/5)g'(y)dy}{g(y)^{n+1}}
  \label{eq:Llnuy_integraldef}
\end{equation}
which we again compute using the Ceva trisectrix \eqref{eq:cevaTheta} as integration contour.
We adopt the same decomposition into central and tail arcs by writing
$L_\ell^{(n)} = L_{\ell;+}^{(n)} + L_{\ell;-}^{(n)} + L_{\ell;\mathrm{tail}}^{(n)}$.

By repeating the analysis of the previous section that led to \eqref{eq:Kl_hl_prop_relation}, 
but now with $L_{\ell}(g,u)$ instead of $K_{\ell}(g,u)$,
we obtain, now by use of \eqref{eq:substhatL},
that for small $\delta$, $y = 1 - \frac{\delta^2}{3}$, $\ell := \lfloor \frac{L}{\delta} \rfloor$, and $t = 1+ \tau \delta$, we have:
\begin{equation}
\hat L(t,g(y),9/5) - \frac{16}{81} \cdot \hat h(t, z(y)) = \frac{1}{9} - \frac{4(2 \pi^2 - 15)^2}{81} \tau^2 \delta^2 + O(\delta^{\frac{7-5b}{3}}).
  \label{eq:hatL_hath_expension}
\end{equation}
Mimicking the proof of Lemma~\ref{lem:subcritTPF} and taking $b=1/16$, this leads to
\begin{equation}
  L_{\ell}(g(y),u) - \frac{16}{81} h_{\ell}(z(y)) = O\left( \delta^{19/6} \right).
  \label{eq:Ll}
\end{equation}
The contribution of $L_{\ell;-}^{(n)}$ is obtained by setting $\theta = - \frac{\pi}{3}+ s \, n^{-1/3}$ with $s \in [0, n^\epsilon]$.
We have in particular
\begin{equation}
  y = y(\theta) = 1 + \sqrt{3} e^{\frac{2 i \pi}{3}} \frac{s}{n^{1/3}} + O(n^{2\epsilon-2/3}),
  \label{eq:y_ceva_central_arc}
\end{equation}
so that
\begin{equation}
  \delta \sim \frac{\sqrt{ 3\sqrt{3} \  e^{-i \pi/3}s}}{n^{1/6}}
\end{equation}
We also consider the scaling $\ell = \left\lfloor d n^{1/6} \right\rfloor$,
so that $L \sim d \sqrt{ 3\sqrt{3} \  e^{-i \pi/3}s}$.
Using \eqref{eq:gyncentralcrit}, \eqref{eq:y_ceva_central_arc} and \eqref{eq:dlog_g_crit_ceva},
we obtain 
\begin{equation}
  \begin{split}
    L_{\ell;-}^{(n)} \sim& \left( \frac{25}{432} \right)^{-n} \frac{1}{2 i \pi} 
    \int_{0}^{n^\epsilon} ds \exp \left( - \frac{3\sqrt{3}}{4} s^3 \right) 
    \frac{9\sqrt{3}}{4n} s^2 \\
    &\hspace{8em}\times \frac{4}{9} 
    \frac{\cosh \!\left( \frac{d}{2} \sqrt{ 3\sqrt{3} \  e^{-i \pi/3}s} \right)}{\sinh^3 \!\left( \frac{d}{2} \sqrt{ 3\sqrt{3} \  e^{-i \pi/3}s} \right)}
    \frac{3^{9/4} s^{3/2} e^{-i\pi/2}}{n^{1/2}}\\
                         &\sim
   - \frac{1}{n^{3/2}} \left( \frac{25}{432} \right)^{-n} \frac{3^{11/4}}{2 \pi} 
    \int_{0}^{\infty} \!\!ds \ s^{7/2} \exp \left( - \frac{3\sqrt{3}}{4} s^3 \right) 
    \frac{\cosh \!\left( \frac{d}{2} \sqrt{ 3\sqrt{3} \  e^{-i \pi/3}s} \right)}{\sinh^3 \!\left( \frac{d}{2} \sqrt{ 3\sqrt{3} \  e^{-i \pi/3}s} \right)}
  \end{split}
\end{equation}
and of course $L_{\ell;+}^{(n)} = \left( L_{\ell;-}^{(n)} \right)^*$,
while $L_{\ell}^{(n)} \sim L_{\ell;+}^{(n)} + L_{\ell;-}^{(n)}$ as again $L_{\ell;\mathrm{tail}}^{(n)}$ is negligible.

\paragraph{Normalization.}
We obtain from the relation $\hat L(1,g,9/5) 
= 1- \frac{1}{r^2(g(y))}= \frac{y(6-y)}{9} = 
\frac{5}{9} - \frac{4}{9}(1-y) + O( (1-y)^2)$
the estimate
\begin{equation}
  \begin{split}
    L^{(n)} \sim& \left( \frac{25}{432} \right)^{\! -n} \!\! \frac{1}{2 i \pi} 
    \int_{0}^{n^\epsilon} \!\! ds \exp \left( - \frac{3\sqrt{3}}{4} s^3 \right) 
    \frac{9\sqrt{3}}{4n} s^2 \!
    \left( \! -\frac{4\sqrt{3}}{9} \frac{s}{n^{1/3}} e^{-i \pi/3} + \frac{4\sqrt{3}}{9} \frac{s}{n^{1/3}} e^{i \pi/3} \right) \\
    \sim& \left( \frac{25}{432} \right)^{\! -n} \!\! \frac{1}{n^{4/3}} 
    \frac{\Gamma(7/3)}{2^{1/3}\sqrt{3}\pi}.
  \end{split}
  \label{eq:hatL_normalization}
\end{equation}
This eventually leads to the block density profile
\begin{equation}
  \begin{split}
    \rho_{\mathrm{block}}(d) =& - \frac{3^{13/4}}{2^{2/3} \Gamma(7/3)} 
    \int_{0}^{\infty} \!\!ds \ s^{7/2} \exp \left( - \frac{3\sqrt{3}}{4} s^3 \right) \\
    & \hspace{6em} \times \left(
      \frac{\cosh \!\left( \frac{d}{2} \sqrt{ 3\sqrt{3} \  e^{-i \pi/3}s} \right)}{\sinh^3 \!\left( \frac{d}{2} \sqrt{ 3\sqrt{3} \  e^{-i \pi/3}s} \right)}
    + \frac{\cosh \!\left( \frac{d}{2} \sqrt{ 3\sqrt{3} \  e^{i \pi/3}s} \right)}{\sinh^3 \!\left( \frac{d}{2} \sqrt{ 3\sqrt{3} \  e^{i \pi/3}s} \right)}
  \right)\\ 
  =& \ \Phi_{\mathrm{block}}'(d)
  \end{split}
  \label{eq:rhoblock_d}
\end{equation}
where, with the variable change $\mu = \frac{\sqrt{3}s}{2^{2/3}}$, 
\begin{equation}
  \begin{split}
    \Phi_{\mathrm{block}}(d) &= \frac{4}{\sqrt{3}\Gamma(7/3)} \int_0^{\infty} d\mu
    \mu^{3} e^{-\mu^{3}}
    \left\{ e^{i\pi/6} \left( 1+ \frac{3}{\sinh^{2}\left( \frac{\sqrt{3}}{2^{2/3}} d \sqrt{\mu} e^{-i\pi/6}\right)} \right) \right. \\
    &\hspace{12em} \left. + e^{-i\pi/6} \left( 1+ \frac{3}{\sinh^{2}\left( \frac{\sqrt{3}}{2^{2/3}} d \sqrt{\mu} e^{i\pi/6}\right)} \right)\right\}\\
  \end{split}
  \label{eq:Phiblock_d}
\end{equation}
so that
\begin{equation}
  \Phi_{\mathrm{block}}(d) 
  = \frac{4}{\Gamma(7/3)} \int_0^{\infty} d\mu \ \mu^{3} e^{-\mu^{3}}
    \left\{ 1 - 6 \frac{1 - \tilde c \breve c + \frac{1}{\sqrt{3}} \tilde s \breve s}{(\tilde c - \breve c)^2} \right\}
\end{equation}
with the shorthand notations of Proposition~\ref{prop:L_ellnasymptcrit}.
By derivation, we obtain the desired result \eqref{eq:rhoblock_d_real}.

\section{Conclusion, discussion}
\label{sec:conc}
It is interesting to confront our results with the predictions of \cite{Fleurat2024}
for the scaling limit of block-weighted quadrangulations.

\paragraph{Supercritical case.}
In \cite[Theorem 5]{Fleurat2024}, it was shown that, as a metric space,
the block-weighted quadrangulation converges to $\sqrt{2}$ times the \emph{Brownian Continuous Random Tree (BCRT)},
up to a scaling factor $\frac{\sigma(u)}{2 \kappa_{u}^{\mathrm{quad}}}n^{-1/2}$
where $\sigma(u)^2 := \frac{3u - 3 + 2\sqrt{u(u-1)}}{5u-9} = \frac{3-\ys}{3(1-\ys)}$
and $\kappa_{u}^{\mathrm{quad}}$ is some appropriate constant.
It is known \cite[Equation~(33)]{Aldous91_CRT_II},
see also \cite[Theorem 3.3.3 with $\gamma=2$]{DuquesneLeGall05_randomtreeslevyprocesses}
that the BCRT has a Rayleigh distance profile
$\rho_{\mathrm{Rayleigh}}(x) = 2x e^{-x^{2}}$,
which corresponds to \eqref{eq:distanceprofile_supercrit} and \eqref{eq:rho_sup_Rayleigh_distance_profile}
up to rescaling $d = D \sqrt{2} \ x$.

Our result is therefore consistent with that of \cite{Fleurat2024} provided that
\begin{equation}
  \kappa_{u}^{\mathrm{quad}} = \frac{\sigma(u)}{\sqrt{2}}D.
\end{equation}
From the expression \eqref{eq:distanceprofile_supercrit} for $D$
and the relation \eqref{eq:uysinv} between $u$ and $\ys$,
it is easily checked that this is equivalent to
\begin{equation}
  \kappa_{u}^{\mathrm{quad}} = \frac{\partial \hat h}{\partial t}(1,\zs) \frac{1}{\hat h^2(1,\zs) \left( 1- \frac{1}{\hat h (1,\zs)} + Q^{\mathrm{sim}}(\zs) \right)},
  \label{eq:kappa_Salvy_in_our_notations}
\end{equation}
by using the identifications $\hat h(1,\zs) = \left( \frac{3}{3-\ys} \right)^2$ and
$Q^{\mathrm{sim}}(\zs) = \frac{\ys(2-\ys)}{3}$.
From the relation $\hat K^{\mathrm{sim}} = 1 - \frac{1}{\hat h}$ we see that
the coefficient $\kappa_{u}^{\mathrm{quad}}$ is identified to
\begin{equation}
  \begin{split}
    \kappa_{u}^{\mathrm{quad}} 
    &= \left. \frac{\partial}{\partial t}\log(\hat K^{\mathrm{sim}}(t,\zs) + Q^{\mathrm{sim}}(\zs)) \right\vert_{t=1} \\
      &= \frac{\sum\limits_{\ell \geq 1} (\ell-1) \left( K_{\ell}^{\mathrm{sim}}(\zs) + \delta_{\ell,1}Q^{\mathrm{sim}}(\zs) \right)}{\sum\limits_{\ell \geq 1}\left( K_{\ell}^{\mathrm{sim}}(\zs) + \delta_{\ell,1}Q^{\mathrm{sim}}(\zs) \right)}
    \end{split}
\end{equation}
which is nothing but the mean of the root distance
of the closest incident vertex of a uniformly chosen edge in a simple rooted quadrangulation
drawn with a weight $\zs$ per face\footnote{
The term $\delta_{\ell,1}Q^{\mathrm{sim}}(\zs)$ corresponds to the situation
where the chosen edge is the root edge itself,
which is not taken into account in $K_{\ell}^{\mathrm{sim}}$.}.
This is fully consistent with the interpretation of $\kappa_{u}^{\mathrm{quad}}$ in \cite[Theorem 5]{Fleurat2024}, the benefit of our approach
is to provide an explicit expression for this quantity.

Recall that, for $u \to \infty$, $\ys \to 0$,
$D \to \frac{1}{2\sqrt{2}}$ (see Figure~\ref{fig:scaling_factors_and_distance_profiles_Kl} and Remark~\ref{rem:scaling_supercrit_limit}), while $\sigma(u) \to 1$.
This gives $\kappa_{u}^{\mathrm{quad}} \to \frac{1}{4}$,
which is also consistent with that interpretation:
as $\zs \to 0$, we select quadrangulations with the smallest number of faces,
hence one face, two edges, and three vertices.
The two such rooted quadrangulations indeed have respectively a mean root distance $0$ and $\frac{1}{2}$,
which averages to $\kappa_{\infty}^{\mathrm{quad}} = \frac{1}{4}$.

\paragraph{Critical case.}
In the critical case $u=9/5$,
\cite[Theorem 5]{Fleurat2024} now predicts, 
with a scaling $\frac{2^{2/3}}{3 \  \kappa_{9/5}^{\mathrm{quad}}}n^{-1/3}$,
that the quadrangulation converges to the $3/2$-stable Lévy tree,
whose density profile can be computed
\cite[Theorem 3.3.3]{DuquesneLeGall05_randomtreeslevyprocesses} as:
\begin{equation}
  \rho_{\mathrm{Levy}}(x) = \frac{3^{5/3}}{4}\Gamma(1/3)x \Airy\left( 3^{2/3}\frac{x}{2} \right).
  \label{eq:DuquesneLevy_profile}
\end{equation}
Setting $x=\frac{2^{2/3}}{3 \  \kappa_{9/5}^{\mathrm{quad}}} d$,
we recognize our result \eqref{eq:distanceprofile_crit} for $\rho_{\mathrm{crit}}$ upon setting
\begin{equation}
  D = 6^{1/3}\kappa_{9/5}^{\mathrm{quad}},
\end{equation}
\emph{id. est.} $\kappa_{9/5}^{\mathrm{quad}} = \frac{2 \pi^2 - 15}{6}$
which is fully consistent with Equation~\eqref{eq:kappa_Salvy_in_our_notations}
at $\ys=1$ for which $\zs=\frac{4}{27}$,
$\hat h(1,\zs) = \frac{9}{4}$ and $Q^{\mathrm{sim}}(\zs) = \frac{1}{3}$,
and $\frac{\partial \hat h}{\partial t}(1,\zs) =\frac{3}{4}(2 \pi^2 - 15)$ 
as seen from the linear term in \eqref{eq:hexptsmall}.

\paragraph{Subcritical case.}
In the subcritical case $u < 9/5$,
we first note that $D=1$ at $u=1$, so that $\rho_{\mathrm{sub}} = \rho_1$.
The scaling factor $D$ in \eqref{eq:rho_scaling_subcrit} can be understood as follows:
it is known that for $u < 9/5$, the map has a unique giant block
of extensive size $k_u = f(u) n$ with $f(u) = \frac{9-5u}{3(3+u)}$.
Moreover, the distance from any point to the giant block
is of negligible order $o(n^{1/4})$ 
compared to the typical distance inside the giant block $\propto k_u^{1/4}$.
It is therefore expected that the two-point function should be exactly the same
when the distance is expressed through the rescaled variable
$\left( \frac{d}{k_u} \right)^{1/4}$.
Since we express the distance in terms of the variable $\left( \frac{d}{n^{1/4}} \right)$,
we expect to obtain the same two-point function upon rescaling the distance by a factor
$\left( \frac{k_u}{n} \right)^{1/4} = \left( f(u) \right)^{1/4}$.
In particular, the passage from $\rho_u$ to $\rho_{\mathrm{sub}} = \rho_1$
simply corresponds to rescaling by a factor $\left( \frac{f(u)}{f(1)} \right)^{1/4}$
which is precisely our factor $D$ in \eqref{eq:rho_scaling_subcrit}.

\paragraph{Block two-point function}
In this paper we showed the computations concerning the block two-point function
only in the critical case.
In the subcritical case, we expect that,
since the root block of a quadrangulation with two marked edges 
is the giant block with probability $1$, with size $f(u)n$,
the block distance profile is clearly equal to
$\frac{1}{\left( f(u) \right)^{1/4}}\rho_0\left(\frac{d}{\left( f(u) \right)^{1/4}}\right)$,
which is nothing but $\rho_u(d)$.

As for the critical case itself,
we expect that the two marked edges are in a ``large'' block of size $k_u = x n^{2/3}$,
with $x$ distributed according to some probability density $\sigma(x)$.
We then expect that, since $k_u^{1/4} = x^{1/4}n^{1/6}$:
\begin{equation}
  \rho_{\mathrm{block}}(d) = \int_{0}^{\infty} \sigma(x)dx \times \frac{1}{x^{1/4}}\rho_0\left(\frac{d}{x^{1/4}}\right).
  \label{eq:rho_block_convolution}
\end{equation}
This identity is checked in \cite{DG26} where an explicit expression for $\sigma(x)$ is obtained.

\appendix

\section{Approximations of $\hat{h}$ near its singularity: proofs}
\label{app:hsing}

In this appendix we prove the approximation results stated in
Section~\ref{ssec:hsingtext} for $\hat h(t,z(y))$ as $y \to 1$.
By~\eqref{eq:hhatmeroalt} and~\eqref{eq:ryexp}, we may write
\begin{equation}
  \label{eq:hphi}
  \hat h(t,z(y)) = \frac{9}{(3-y)^2} \left( 1 - 4 (1-t) \phi(t,x(y)) \right)
\end{equation}
where we introduce the simpler auxiliary function
\begin{equation}
  \phi(t,x) := \frac{(1-x^2)^2}{4x^2} \sum_{k=1}^\infty \frac{k x^{3k}}{1-t x^{k}}.
  \label{eq:phidef}
\end{equation}
Our task boils down to understanding the behavior of $\phi(t,x)$ for
$x \to 1$. Note that, reasoning as in~\eqref{eq:hhatmero} backwards,
or alternatively using again~\eqref{eq:kl_simpleelements} in the
series expansion of $\hat h$, we find that $\phi(t,x)$ expands as a
series in $t$ as
\begin{equation}
  \label{eq:phiser}
  \phi(t,x) = \frac{(1+x)^2}{4x^2} \sum_{n=0}^\infty \frac{x^{n+3}(1-x)^2}{(1-x^{n+3})^2} \cdot t^n.
\end{equation}

\paragraph{The fixed $t$ regime.} We start by the case where $t$
remains fixed as $x\to 1$. It is instructive to consider first the
case $|t| \leq 1$: by taking the termwise limit in~\eqref{eq:phiser}
we see that $\phi(t,x)$ tends to the shifted dilogarithm $\Phi(t)$
of~\eqref{eq:Htexp}. To justify the exchange of limit and summation,
we observe that, for any integer $m$, we have
\begin{equation}
  \left\lvert \frac{x^{m/2}}{1-x^{m+1}} \right\rvert \leq
  \frac{|x|^{m/2}}{1-|x|^{m+1}} = \frac{1}{1-|x|} \cdot \frac{1}{|x|^{-m/2}+\cdots+|x|^{m/2}} \leq \frac{1}{1-|x|} \cdot \frac{1}{m+1}
\end{equation}
where the last inequality results from $|x|^k+|x|^{-k} \geq 2$ for any
$k=-m/2,\ldots,m/2$. Applying this with $m=n+2$, we find
\begin{equation}
  \left\lvert \phi(t,x) \right\rvert \leq
  \frac{|1+x|^2}{4|x|} \left( \frac{|1-x|}{1-|x|} \right)^2 \sum_{n=0}^\infty \frac{|t|^n}{(n+3)^2}.
\end{equation}
The dominated convergence theorem then tells us that
$\phi(t,x) \to \Phi(t)$ provided that $x$ tends to $1$ in such a way
that $\frac{|1-x|}{1-|x|}$ remains bounded, i.e.\ non tangentially to
the unit circle. This is the case under the assumption of
Section~\ref{ssec:hsingtext} that
$\arg \delta \to \alpha \in (-\frac\pi2,\frac\pi2)$, since
$1-x(y) \sim \delta$. Note that this reasoning gives a direct proof of
the observation made in Remark~\ref{rem:Htexp} that
$\hat h(t,z(y)) \to H(t)$ for $|t|<1$.

Let us now consider the general situation where $t$ is possibly
outside the unit circle, in which case we should proceed from the
meromorphic form~\eqref{eq:phidef} of $\phi(t,x)$. It is convenient to
write
\begin{equation}
  x = e^{-\omega \ve}
\end{equation}
with $\omega$ having modulus $1$ and $\ve$ real and
positive. Note that $\omega \ve=\delta+O(\delta^2)$, so the
limit considered in Section~\ref{ssec:hsingtext} corresponds to
$\ve \to 0^+$, $\omega \to e^{i\alpha}$. It is plain that, in
this limit, the sum in $\phi(t,x)$ yields an integral, namely
\begin{equation}
  \label{eq:phitxlim}
  \phi(t,x) \sim \omega^2 \ve^2 \sum_{k=1}^\infty \frac{k e^{-3\omega\ve k}}{1-t e^{-\omega \ve k}} \to e^{2i\alpha} \int_0^\infty dv \frac{v e^{-3e^{i\alpha} v}}{1-t e^{-e^{i\alpha} v}} = \Phi_\alpha(t)
\end{equation}
where $\Phi_\alpha(t)$ is as in~\eqref{eq:hexptfin}, as seen via the
change of variable $w=e^{i\alpha}v$. Of course, we shall assume that
$t$ does not belong to the curve
$\mathcal S_\alpha:=\{ e^{e^{i\alpha} v}, v > 0\}$ so that the
integral above does not pass through a pole. Note that we may still
take $t=1$ as the integral is regular at $v=0$.

For our purposes we need to control the error when approximating
$\phi(t,x)$ by $\Phi_\alpha(t)$. To this end, we use the
Euler-Maclaurin formula at first order, which asserts that
\begin{equation}
  \label{eq:EulerMaclaurin}
  \sum_{k=0}^\infty f(k) = \int_0^\infty f(k) dk + \frac{f(0)}{2} + \int_0^\infty f'(k) \left( k - \lfloor k \rfloor - \frac 12 \right) dk
\end{equation}
whenever $f$ is a differentiable function with sufficient decay at
infinity. Here we apply the formula with
$f(k)=\omega^2 \ve^2 k \frac{e^{-3\omega \ve k}}{1-t e^{-\omega \ve
    k}}$, so we have $f(0)=0$. The first integral is equal to
$\Phi_{\arg\omega}(t)$, which is equal to $\Phi_\alpha(t)$ provided
that $\omega$ is close enough to $e^{i\alpha}$ so that the curve
$\mathcal S_{\arg \omega}$ can be deformed into $\mathcal S_\alpha$
without crossing $t$. As for the rightmost integral, corresponding to
the error term, it may be bounded by
\begin{equation}
  \label{eq:EMLerrorbound}
  \frac12 \int_0^\infty |f'(k)| dk = \frac{\ve}2 \int_0^\infty \left\lvert
    \frac{d}{dv} \frac{v e^{-3\omega v}}{1-te^{-\omega v}} \right\vert dv = O(\ve).
\end{equation}
From this we conclude that
\begin{equation}
  \label{eq:phiestim1}
  \phi(t,x) = \Phi_\alpha(t) + O(\ve)
\end{equation}
where the error is uniform for $t$ in any compact subset of
$\C \setminus \mathcal S_\alpha$.

Interestingly, we may even control the error for $t$ approaching
$\mathcal S_\alpha$ slowly enough. Namely, the integrand in the
right-hand side of~\eqref{eq:EMLerrorbound} involves a denominator
$|1-te^{-\omega v}|^2$. Upon integrating, we find that the error is
bounded by a constant times
$\sup_{v \geq 0} \frac{\ve}{|1-t e^{-\omega v}|}$. Assume that
$\omega$ tends to $e^{i\alpha}$ regularly, i.e.
$\omega = e^{i\alpha}+O(\ve)$, and fix $R>1$: then for $|t| \leq R$ we
have
$\sup_{v \geq 0} \frac{\ve}{|1-t e^{-\omega v}|} =
\Theta(\ve/d_\alpha(t))$ where $d_\alpha(t)$ denotes the distance
between $t$ and $\mathcal S_\alpha$. For our purposes, we only need
this control for $t$ approaching $1$: assuming that $\arg(t-1)$
remains away from $\alpha$ we have $d_\alpha(t) = \Theta(|t-1|)$ so we
find that the error in~\eqref{eq:phiestim1} is a $O(\ve/|t-1|)$.
Plugging these estimates for $\phi(t,x)$ into~\eqref{eq:hphi} yields
Proposition~\ref{prop:hexptfin}, upon recalling that
$\ve \sim |\delta|$.

\paragraph{The regime $t$ close to $1$.} We now analyze the regime
where $t-1$ is of order $\ve$. The previous discussion
provides little information here. It is convenient to write
\begin{equation}
  t = x^{-\tau} = e^{\tau \omega \epsilon}
\end{equation}
with $\tau$ not a positive integer, so that
\begin{equation}
  \label{eq:phit1sum}
  \phi(t,x) = \frac{(\sinh \omega \ve)^2}{(\omega \ve)^2} F_{\omega\ve}(\tau), \qquad
  F_{\omega\ve}(\tau) := \omega^2 \ve^2  \sum_{k=1}^\infty \frac{k e^{-3\omega\ve k}}{1-e^{-\omega \ve (k-\tau)}}.
\end{equation}
At the zeroth order in $\ve$, we may approximate the sum by an
integral as in~\eqref{eq:phitxlim}, and find that
$F_{\omega\ve}(\tau)$, hence $\phi(t,x)$, tend to
$\Phi_\alpha(1)=\frac{2\pi^2-15}{12}$. We want however more, namely to
keep track of the dependence on $\tau$, which requires to push the
expansion in $\ve$ further. This poses a new difficulty: on the one
hand, if we expand the summand in~\eqref{eq:phit1sum} as a series in
$\tau$ and try approximating sums by integrals as before, the latter
diverge at $0$; on the other hand, if we just take the termwise limit
$\ve \to 0$, we get a divergent sum. To overcome this difficulty, we
introduce a cutoff $K$, such that $1 \ll K \ll \ve^{-1}$, and split the sum
in~\eqref{eq:phit1sum} in two: let us denote by
$F_{\omega\ve}^{(<K)}(\tau)$ and $F_{\omega\ve}^{(\geq K)}(\tau)$ the
contributions to $F_{\omega\ve}(\tau)$ coming from the terms
$k=1,\ldots,K-1$, and $k=K,K+1,\ldots$, respectively. These two
contributions will be analyzed by different methods: termwise limit in
the sum for the former, approximation by integrals for the latter.

Before attacking the computations, let us also mention that we would
like to obtain a uniform control on the error, even for $\tau$
large. To this end, we introduce a secondary scale $T$, also possibly
depending on $\ve$ and satisfying $1 \leq T \ll K$, and we assume that
\begin{equation}
  \label{eq:tauconst}
  |\tau| \leq T, \qquad \min_{k \in \Z_{>0}} |\tau-k| \geq \vr
\end{equation}
where $\vr \in (0,\frac12)$ is arbitrary and fixed. (We could in
principle allow the third scale $\vr$ to vary with $\ve$ but this is
not needed for our purposes.)

We now analyze $F_{\omega\ve}^{(<K)}(\tau)$ by taking the termwise limit in the
sum. First, by a Taylor expansion we have for any $k=1,\ldots,K-1$,
\begin{equation}
  \frac{\omega \ve (k-\tau)}{1-e^{-\omega \ve (k-\tau)}} = 1 + O(\ve K), \qquad
  e^{-3\omega \ve k} = 1 - O(\ve K)
\end{equation}
with uniform error terms. Multiplying these two expansions together
and by $\frac{k}{k-\tau}$, which is a $O(T)$ as it is
largest for $k \sim T$ and $k-\tau=\Theta(1)$, we get
\begin{equation}
  \frac{\omega \ve k e^{-3\omega\ve k}}{1-e^{-\omega \ve (k-\tau)}} = \frac{k}{k-\tau} + O(\ve K T)
\end{equation}
and thus, by summing over $k$,
\begin{equation}
  \sum_{k=1}^{K-1} \frac{\omega \ve k e^{-3\omega\ve k}}{1-e^{-\omega \ve (k-\tau)}} = \sum_{k=1}^{K-1} \frac{k}{k-\tau} + O(\ve K^2 T).
\end{equation}
Now, introducing the digamma function $\psi$, which
satisfies $\psi(v+1)-\psi(v)=\frac1v$, the harmonic sum may be
rewritten as
\begin{equation}
  \sum_{k=1}^{K-1} \frac{k}{k-\tau} = K-1 + \tau \left(\psi(K-\tau)-\psi(1-\tau)\right) =
  K + \tau \psi(K-\tau)- \tau \psi(-\tau).
\end{equation}
Multiplying by $\omega \ve$ and using finally that $\psi(K-\tau)=\ln K +O(K^{-1})$, we obtain
\begin{equation}
  \label{eq:FmKesti}
  F_{\omega\ve}^{(<K)}(\tau) = \left( K + \tau \ln K - \tau \psi(-\tau) \right) \omega \ve + O(\ve^2 K^2 T).
\end{equation}

We turn to the analysis of $F_{\omega\ve}^{(\geq K)}(\tau)$ where we will use
again the Euler-Maclaurin formula applied to the function
$f(k)=\omega^2 \ve^2 k \frac{e^{-3\omega \ve k}}{1-t e^{-\omega \ve
    k}}$, but now pushed to the third order:
\begin{equation}
  \label{eq:EulerMaclaurinthird}
  \sum_{k=K}^\infty f(k) = \int_K^\infty f(k) dk + \frac{f(K)}{2} - \frac{f'(K)}{12} + \int_K^\infty f'''(k) \frac{B_3(k - \lfloor k \rfloor)}{6}dk.
\end{equation}
Here, $B_3(v)=v^3-\frac32 v^2+\frac12 v$ is the Bernoulli polynomial
of order $3$. Let us analyze each term in the formula. By a change of
variable $v=\ve k$, the first integral reads
\begin{equation}
  \int_K^\infty f(k) dk = \omega^2 \int_{\ve K}^\infty dv \frac{v e^{-3\omega v}}{1-t e^{-\omega v}}.
\end{equation}
Since $t=1+O(\ve T)$, we may expand the integrand as
\begin{equation}
  \label{eq:firstintser}
  \frac{v e^{-3\omega v}}{1-t e^{-\omega v}} = \frac{v e^{-3\omega v}}{1-e^{-\omega v}} + (t-1) \frac{v e^{-4\omega v}}{(1-e^{-\omega v})^2} + O\left( \ve^2 T^2 \frac{v e^{-5\omega v}}{(1-e^{-\omega v})^3} \right).
\end{equation}
Integrating each of these terms in $v$ from $\ve K$ to $\infty$ we get
\begin{equation}
  \begin{split}
    \int_{\ve K}^\infty dv \frac{v e^{-3\omega v}}{1-e^{-\omega v}} &=
    \underbrace{\frac{2\pi^2-15}{12}}_{\Phi(1)} - \omega \ve K + O(\ve^2 K^2), \\
    (t-1) \int_{\ve K}^\infty dv \frac{v e^{-4\omega v}}{(1-e^{-\omega v})^2} &= \left(\omega \ve \tau + O(\ve^2 T^2)\right) \left( - \ln (\omega \ve K) + \frac{13-2 \pi^2}4 + O(\ve K) \right) \qquad \\
    &= \omega \ve \tau \left( - \ln (\omega \ve K) + \frac{13-2 \pi^2}4  \right) + O(\ve^2 K T, \ve^2 T^2 \ln \ve ),
  \end{split}
\end{equation}
while the last error term in~\eqref{eq:firstintser} yields a
$O(\ve^2 T^2 (\ve K)^{-1}) = O(\ve K^{-1} T^2)$ since the integrand
diverges as $v^{-2}$ for $v \to 0$, so the integral diverges as the
inverse of its lower bound.

Now, the middle terms in~\eqref{eq:EulerMaclaurinthird} read
\begin{equation}
  \begin{split}
    \frac{f(K)}2 &= \frac{\omega^2 \ve^2 K e^{-3 \omega \ve K}}{2(1-e^{-\omega \ve(K-\tau)})} =
                   \frac{\omega \ve}2 \cdot \frac{K}{K-\tau} \cdot \frac{\omega \ve (K-\tau)}{1-e^{-\omega \ve(K-\tau)}} \cdot e^{-3\omega \ve K} \\
    &=
      \frac{\omega \ve}{2} \left(1+O(K^{-1}T)\right) (1+O(\ve K)) =
      \frac{\omega \ve}2 + O(\ve K^{-1}T, \ve^2 K), \\
    \frac{f'(K)}{12} &= O(\ve K^{-2} T)
  \end{split}
\end{equation}
and the modulus of the last error term in~\eqref{eq:EulerMaclaurinthird} may be bounded by a constant times
\begin{equation}
  \int_K^\infty |f'''(k)| dk =
  \ve^3 \int_{\ve K}^\infty \left\lvert
    \frac{d^3}{dv^3} \frac{v e^{-3\omega v}}{1-te^{-\omega v}} \right\vert dv.
\end{equation}
Since the integral is on $v \geq \ve K \gg t-1$, we may estimate the
right-hand side by substituting $t=1$, and noting that the integrand
diverges as $v^{-3}$ for $v \to 0$, so the integral diverges as
$(\ve K)^{-2}$: multiplying by $\ve^3$ we conclude that the error term
is a $O(\ve K^{-2})$.

By collecting everything, we arrive at
\begin{multline}
  F_{\omega\ve}^{(\geq K)}(\tau) = \frac{2\pi^2-15}{12} + \left( -K - \tau \ln(\omega \ve K) + \frac{13-2 \pi^2}4 \tau + \frac12 \right) \omega \ve  \\
  + O(\ve^2 K^2,\ve^2 T^2 \ln \ve,\ve K^{-1} T^2)
\end{multline}
where we have used the comparison of scales
$1 \leq T \ll K \ll \ve^{-1}$ to identify the possible largest
errors. Adding to~\eqref{eq:FmKesti} we see that the dependence in $K$
for the non-error terms disappears, as it should, to yield
\begin{multline}
  F_{\omega\ve}(\tau) = \frac{2\pi^2-15}{12} -\tau \omega \ve \ln(\omega \ve)+  \left( - \tau \psi(-\tau) + \frac{13-2 \pi^2}4 \tau + \frac12 \right) \omega \ve \\
  + O(\ve^2 K^2 T,\ve^2 T^2 \ln \ve,\ve K^{-1} T^2).
\end{multline}
Now, we may optimize the error by choosing the cutoff $K$ in such a
way that $\ve^2 K^2 T \sim \ve K^{-1} T^2$, i.e.
$K \sim (T/\ve)^{1/3}$, so that the error is $O(\ve^{4/3}
T^{5/3})$. We see that we should assume
\begin{equation}
  T \ll \ve^{-1/5}
\end{equation}
for this error to be negligible with respect to the terms of order
$\ve$.

In order to connect with Section~\ref{ssec:hsingtext}, let us rewrite
our result in terms of the first auxiliary function $\phi(t,x)$ and of
the variable $\delta=\omega \ve+O(\ve^2)$.  Note that neither the
proportionality factor
$(\sinh \omega \ve)^2/(\omega \ve)^2=1+O(\delta)^2$ appearing in
\eqref{eq:phit1sum}, nor the redefinition $t=1+\tau \delta$ (replacing
$t=x^{-\tau}$) alter the expansion at first order.

\begin{prop}
  Given $\tau \in \C \setminus \Z_{>0}$, we have as $\delta \to 0$
  \begin{equation}
    \phi\left(1+\tau \delta,x(1-\delta^2/3)\right) 
    = \frac{2\pi^2-15}{12} - \tau \delta \ln \delta+ \left( - \tau \psi(-\tau) + \frac{13-2 \pi^2}4 \tau + \frac12 \right) \delta + O(\delta^{4/3})
  \end{equation}
  where the error term $O(\delta^{4/3})$ is uniform for $\tau$ in any
  compact subset of $\C \setminus \Z_{>0}$. It becomes a uniform
  $O(\delta^{(4-5b)/3})$ for all $\tau$ satisfying~\eqref{eq:tauconst}
  with $T=\delta^{-b}$, $b \in (0,\frac15)$.
 \end{prop}

Proposition~\ref{prop:hexptsmall} is an immediate corollary of this
proposition and of~\eqref{eq:hphi}.

\printbibliography
\end{document}